\documentclass{amsart}

\usepackage{amscd}
\usepackage{amsmath}
\usepackage{amssymb}
\usepackage{amsthm}
\usepackage{arydshln}
\usepackage{bbm}
\usepackage{bm}
\usepackage{colonequals}
\usepackage{color}
\usepackage{enumitem}
\usepackage{latexsym}
\usepackage{mathrsfs}
\usepackage{mathtools}
\usepackage{stmaryrd}
\usepackage{tikz-cd}
\usepackage{varwidth}
\usepackage[all]{xy}
\usepackage{yhmath}

\usepackage{url}

\newcommand{\Art}{\mathrm{Art}}

\newcommand{\Std}{\mathrm{Std}}

\newcommand{\C}{\mathbb{C}}
\newcommand{\F}{\mathbb{F}}
\newcommand{\Q}{\mathbb{Q}}
\newcommand{\R}{\mathbb{R}}
\newcommand{\Z}{\mathbb{Z}}

\newcommand{\G}{\mathbf{G}}

\newcommand{\mfd}{\mathfrak{d}}

\newcommand{\mfp}{\mathfrak{p}}

\newcommand{\mcO}{\mathcal{O}}

\newcommand{\mcS}{\mathcal{S}}

\DeclareMathOperator{\slope}{slope}

\DeclareMathOperator{\val}{val}

\DeclareMathOperator{\Ad}{Ad}

\DeclareMathOperator{\Frob}{Frob}
\DeclareMathOperator{\Gal}{Gal}
\DeclareMathOperator{\Hom}{Hom}

\DeclareMathOperator{\Ind}{Ind}

\DeclareMathOperator{\Irr}{Irr}

\DeclareMathOperator{\Ker}{Ker}
\DeclareMathOperator{\Lie}{Lie}
\DeclareMathOperator{\LLC}{LLC}

\DeclareMathOperator{\Nr}{Nr}

\DeclareMathOperator{\Sw}{Sw}
\DeclareMathOperator{\Sym}{Sym}
\DeclareMathOperator{\Tr}{Tr}

\DeclareMathOperator{\GL}{GL}

\DeclareMathOperator{\SL}{SL}
\DeclareMathOperator{\SO}{SO}
\DeclareMathOperator{\Sp}{Sp}

\theoremstyle{plain}
\newtheorem{thm}{Theorem}[section]
\newtheorem*{thm*}{Theorem}
\newtheorem{prop}[thm]{Proposition}
\newtheorem{lem}[thm]{Lemma}
\newtheorem{cor}[thm]{Corollary}
\newtheorem{conj}[thm]{Conjecture}

\theoremstyle{definition}

\theoremstyle{remark}
\newtheorem{rem}[thm]{Remark}

\newtheorem*{claim*}{Claim}

\SelectTips{cm}{11}

\title{On Swan exponents of symmetric and exterior square Galois representations}

\author{Guy Henniart}
\address{Universit\'e Paris-Saclay, CNRS, Laboratoire de Math\'ematiques d'Orsay, 91405, Orsay, France.}
\email{Guy.Henniart@math.u-psud.fr}

\author{Masao Oi}
\address{Department of Mathematics (Hakubi center), Kyoto University, Kitashirakawa, Oiwake-cho, Sakyo-ku, Kyoto 606-8502, Japan.}
\email{masaooi@math.kyoto-u.ac.jp}

\dedicatory{Dedicated to Marko Tadi\'{c}, on his 70th birthday.$^{\ast}$}

\thanks{$^{\ast}$The first author met Marko Tadi\'{c} already in 1971, at the International Mathematical Olympiads in Slovakia. But we realized that only later, when discussing mathematics and more, which happened on many nice occasions.}

\begin{document}

\begin{abstract}
Let $F$ be a local non-Archimedean field and $E$ a finite Galois extension of $F$, with Galois group $G$.
If $\rho$ is a representation of $G$ on a complex vector space $V$, we may compose it with any tensor operation $R$ on $V$, and get another representation $R\circ\rho$. 
We study the relation between the Swan exponents $\Sw(\rho)$ and $\Sw(R\circ\rho)$, with a particular attention to the cases where $R$ is symmetric square or exterior square. 
Indeed those cases intervene in the local Langlands correspondence for split classical groups over $F$, via the formal degree conjecture, and we present some applications of our work to the explicit description of the Langlands parameter of simple cuspidal representations.
For irreducible $\rho$ our main results determine $\Sw(\Sym^{2}\rho)$ and $\Sw(\wedge^{2}\rho)$ from $\Sw(\rho)$ when the residue characteristic $p$ of $F$ is odd, and bound them in terms of $\Sw(\rho)$ when $p$ is $2$. 
In that case where $p$ is $2$ we conjecture stronger bounds, for which we provide evidence.
\end{abstract}

\maketitle

\tableofcontents

\section{Introduction}

Let $p$ be a prime number, $F$ a non-Archimedean local field of residue characteristic $p$, and $E$ a finite Galois extension of $F$ with Galois group $G=\Gal(E/F)$.
Suppose that $\rho$ is a finite-dimensional representation of $G$ on a complex vector space $V$ (in this paper, we refer to such $\rho$ simply as a \textit{Galois representation}).
For any algebraic representation $R$ of $\GL_{\C}(V)$ on a finite-dimensional complex vector space $V'$, by composing $R$ with $\rho$, we obtain another Galois representation $R\circ\rho$ of $G$.
\[
\xymatrix{
G\ar_-{\rho}[r] \ar@/^15pt/^-{R\circ\rho}[rr]&\GL_{\C}(V) \ar_-{R}[r]&\GL_{\C}(V')
}
\]

Given that the \textit{Swan exponent} is one of the most fundamental invariants of Galois representations (see Section \ref{subsec:Swan} for the definition), it is natural to ask if there is an interesting relationship between the Swan exponents $\Sw(R\circ\rho)$ and $\Sw(\rho)$. 
This is the problem of our interest.

We first introduce several preceding results on this problem.
Bushnell--Henniart--Kutzko investigated the case where $\rho$ is irreducible and $R=\Std\otimes\Std^{\vee}$, i.e., $R\circ\rho=\rho\otimes\rho^{\vee}$, where $\rho^{\vee}$ denotes the contragredient representation of $\rho$, in \cite{BHK98-JAMS}.
They established an explicit formula of $\Sw(\rho\otimes\rho^{\vee})$.
(In fact, their result is more general; for any two irreducible Galois representations $\rho_{1}$ and $\rho_{2}$, they obtained an explicit formula of $\Sw(\rho_{1}\otimes\rho_{2})$).
Based on the result of Bushnell--Henniart--Kutzko, Anandavardhanan--Mondal (\cite{AM16}) established a similar explicit formula in the case where $R$ is the symmetric square $\Sym^{2}$, the exterior square $\wedge^{2}$, or the Asai representation ``$\mathrm{Asai}$'', by assuming that $p\neq2$.

However, in these cases, the Swan exponent $\Sw(R\circ\rho)$ cannot be expressed only in terms of $\Sw(\rho)$.
The description requires the information of a simple stratum (in the sense of Bushnell--Kutzko, \cite{BK93}) associated to the supercuspidal representation corresponding to $\rho$ under the local Langlands correspondence (cf.\ \cite[6.5, Theorem]{BHK98-JAMS} and \cite[Theorem 6.1]{AM16}).

Based on this observation, we settle for looking for bounds estimating $\Sw(R\circ\rho)$ from $\Sw(\rho)$, sometimes even getting exact formulas.
Concerning this direction, let us mention the work of Bushnell--Henniart \cite{BH17-Iran} and also the work of Kili\c{c} \cite{Kil20}, in which upper bounds and lower bounds of $\Sw(\rho_{1}\otimes\rho_{2})$ are investigated.
For example, some of the inequalities proved in \textit{loc.\ cit.} are as follows:

\begin{thm}[{\cite[Theorems AS and CS]{BH17-Iran}}]
Let $\rho_{1}$ and $\rho_{2}$ be Galois representations whose dimensions are $n_{1}$ and $n_{2}$, respectively.
\begin{enumerate}
\item
If $\rho_{1}$ is minimal in the sense that $\Sw(\rho_{1})\leq\Sw(\rho_{1}\otimes\chi)$ for any character $\chi$, then we have
\[
\Sw(\rho_{1}\otimes\rho_{2})
\geq
\frac{1}{2}\max\{n_{2}\Sw(\rho_{1}),n_{1}\Sw(\rho_{2})\}.
\]
\item
We have
\[
\Sw(\rho_{1}\otimes\rho_{2})
\leq
n_{2}\Sw(\rho_{1})+n_{1}\Sw(\rho_{2})-\min\{\Sw(\rho_{1}),\Sw(\rho_{2})\}.
\]
\end{enumerate}
\end{thm}

\begin{thm}[{\cite[Theorem 9.3]{Kil20}}]
If $\rho$ is a minimal Galois representation on $V$, then, for any algebraic representation $R$ of $\GL_{\C}(V)$, we have
\[
\Sw(\rho\otimes(R\circ\rho))
\geq
\frac{1}{2}\dim(R\circ\rho)\cdot\Sw(\rho).
\]
\end{thm}

Now let us explain our main results.
Our principal interest is in the case where $R=\Sym^{2}$ or $R=\wedge^{2}$.
Our first main result is the following:

\begin{prop}[Proposition \ref{prop:p-odd}]\label{prop:p-odd-intro}
Suppose that $p\neq2$.
For any Galois representation $\rho$, we have
\[
\Sw(\Sym^{2}\rho)-\Sw(\wedge^{2}\rho)
=
\Sw(\rho).
\]
\end{prop}

Note that if $\rho$ is self-dual, then we have $\rho\otimes\rho^{\vee}\cong\rho\otimes\rho\cong\Sym^{2}\rho\oplus\wedge^{2}\rho$.
Thus, by combining Proposition \ref{prop:p-odd-intro} with the explicit formula of $\Sw(\rho\otimes\rho)$ by Bushnell--Henniart--Kutzko mentioned above, it is possible to get a precise formula of $\Sw(\Sym^{2}\rho)$ and $\Sw(\wedge^{2}\rho)$.
Especially, when $\rho$ furthermore satisfies some additional conditions, the resulting formula is particularly simple as follows:

\begin{cor}[Corollary \ref{cor:p-odd}]\label{cor:p-odd-intro}
Suppose that $p\neq2$.
Let $\rho$ be a $2n$-dimensional irreducible self-dual Galois representation such that $(\Sw(\rho),2n)=1$.
Then we have
\[
\Sw(\Sym^{2}\rho)=n\cdot\Sw(\rho)
\quad\text{and}\quad
\Sw(\wedge^{2}\rho)=(n-1)\cdot\Sw(\rho).
\]
\end{cor}

In fact, the equality of Proposition \ref{prop:p-odd-intro} has been already obtained in \cite{AM16}, thus not new.
Our contribution is that we gave an alternative and simpler proof.
See Remark \ref{rem:AM16} for more comments on Proposition \ref{prop:p-odd-intro} and Corollary \ref{cor:p-odd-intro}.

We note that the assumption that $p\neq2$ is crucially used in the proof of Proposition \ref{prop:p-odd-intro}.
In fact, the equality of Proposition \ref{prop:p-odd-intro} does not hold in general when $p=2$.
For example, we have the following:
\begin{prop}[Proposition \ref{prop:ssc-p-2}]\label{prop:ssc-p-2-intro}
Suppose that $p=2$.
Let $\rho$ be a $(2n+1)$-dimensional irreducible self-dual Galois representation such that $(\Sw(\rho),2n+1)=1$.
Then we have
\[
\Sw(\Sym^{2}\rho)-\Sw(\wedge^{2}\rho)=0.
\]
\end{prop}

Hence, by the same argument as in the case where $p\neq2$ (i.e., using the Bushnell--Henniart--Kutzko formula), we obtain the following:

\begin{cor}[Corollary \ref{cor:ssc-p-2}]
With the same assumptions as in Proposition \ref{prop:ssc-p-2-intro}, we have
\[
\Sw(\Sym^{2}\rho)
=\Sw(\wedge^{2}\rho)
=n\cdot\Sw(\rho).
\]
\end{cor}

On the other hand, there is also an example such that $\Sw(\Sym^{2}\rho)-\Sw(\wedge^{2}\rho)$ is not zero (for example, take $\rho$ to be a $1$-dimensional character such that $\rho^{2}$ is wildly ramified).
Then what can we expect about the quantity $\Sw(\Sym^{2}\rho)-\Sw(\wedge^{2}\rho)$ when $p=2$?
Our second main result is the following:

\begin{thm}[Theorem \ref{thm:p-2-weak-ineq}]\label{thm:p-2-weak-ineq-intro}
Suppose that $p=2$.
For any Galois representation $\rho$, we have
\[
0
\leq
\Sw(\Sym^{2}\rho)-\Sw(\wedge^{2}\rho)
\leq
2\Sw(\rho).
\]
\end{thm}

Note that the left inequality of Theorem \ref{thm:p-2-weak-ineq-intro} is ``sharp'' in the sense that there exists $\rho$ which attains the equality (Proposition \ref{prop:ssc-p-2-intro}; see also Proposition \ref{prop:weak-strict}).
On the other hand, the right inequality is not sharp whenever $\rho$ is not tame (Proposition \ref{prop:weak-strict}).
In fact, we expect that the following inequality holds and is sharp.

\begin{conj}\label{conj:Henniart-intro}
Assume $p=2$.
For any Galois representation $\rho$, we have
\[\label{ineq:conj-intro}
\Sw(\Sym^{2}\rho)-\Sw(\wedge^{2}\rho)
\leq
\Sw(\rho).
\tag{$\star$}
\]
\end{conj}

In this paper, we verify Conjecture \ref{conj:Henniart-intro} in several special cases and also discuss attempts to the conjecture in the general case.
\begin{enumerate}
\item
In Section \ref{subsec:2-dim}, we prove the inequality \eqref{ineq:conj-intro} for any $2$-dimensional $\rho$.
\item
In Section \ref{subsec:cyclic}, we prove the inequality \eqref{ineq:conj-intro} for any $\rho$ of the form $\Ind_{F'/F}\chi$, where $F'/F$ is cyclic and $\chi$ is a character of $F^{\prime\times}$.
\item
In Section \ref{subsec:approach-ramif-group}, we discuss an approach based on an analysis of the structure of the ramification subgroups.
We also obtain definitive results in some specific cases (Propositions \ref{prop:I-order-4} and \ref{prop:I-equal-Z-and-H-irr}).
\item
In Section \ref{subsec:induction-order}, we discuss an approach via induction on the order of the Galois group $G$.
Especially, we show that if the conjecture is true, then the inequality \eqref{ineq:conj-intro} must be strict when $\rho$ is orthogonal (\ref{prop:orth-strict-inequality}).
\item
In Section \ref{subsec:induction-dimension}, we discuss an approach via induction on $\dim(\rho)$.
\end{enumerate}

We finally mention that, in fact, our motivation originates from investigating the \textit{explicit local Langlands correspondence for classical groups}.
When $\G$ is a connected reductive group over $F$, the local Langlands correspondence attaches to a smooth irreducible representation $\pi$ of $\G(F)$ an $L$-parameter $\phi$ ($L$-parameters are a variant of Galois representations, see Section \ref{subsec:LLC}).
When $\pi$ is discrete series, it is expected that a certain explicit identity is satisfied between the formal degree of $\pi$ and the special value of the adjoint $\gamma$-factor of $\phi$ (the \textit{formal degree conjecture}, see Section \ref{subsec:FDC}).
The point is that, when $\G$ is a classical group, the $L$-parameter $\phi$ can be naturally regarded as a Galois representation (say $\rho$) and then the adjoint $\gamma$-factor of $\phi$ is related to the Swan exponent $\Sw(R\circ\rho)$ for some tensor operation $R$.
For example, if $\G=\SO_{2n+1}$ (resp.\ $\Sp_{2n}$ or $\SO_{2n}$), then $R$ is given by $\Sym^{2}$ (resp.\ $\wedge^{2}$).
\[
\xymatrix{
\text{$\pi$: irr.\ sm.\ rep.\ of $\G(F)$}\ar@{<->}^-{\LLC}[rr]\ar@{~>}[d]&&\text{$\phi$: $L$-par.} \ar@{~>}[d] \ar@{<->}[r]&\text{$\rho$: Gal.\ rep.} \ar@{~>}[d]\\
\text{formal degree of $\pi$} \ar@{<->}^-{\mathrm{FDC}}[rr]&& \text{adjoint $\gamma$-factor of $\phi$} \ar@{<->}[r] & \Sw(R\circ\rho)
}
\]
Thus, for a given $L$-parameter $\phi$, computing $\Sw(R\circ\rho)$ is related to knowing the formal degree of $\pi$, which can be a clue to determine the representation $\pi$ exactly (and vice versa).
In Section \ref{sec:LLC}, we present a few results which are obtained based on this philosophy.

\medbreak
\noindent{\bfseries Acknowledgment.}\quad 
The authors would like to thank Gordan Savin for helpful suggestions on the content of Section \ref{sec:G2}.
This paper was written as the first author was enjoying the hospitality of the Graduate School of Mathematical Sciences of the University of Tokyo. 
He also thanks the Department of Mathematics of Kyoto University for an invitation.
The second author was supported by JSPS KAKENHI Grant Number 20K14287.

\medbreak
\noindent{\bfseries Notation.}\quad 
For a non-Archimedean local field $F$, we write $\mcO_{F}$, $\mfp_{F}$, and $k_{F}$ for the ring of integers, the maximal ideal, and the residue field of $F$, respectively.
For any $r\in\Z_{>0}$, we let $U_{F}^{r}:=1+\mfp_{F}^{r}$.

\section{Main Results}\label{sec:main}

\subsection{Swan exponent of Galois representations}\label{subsec:Swan}
We first recall the definition of the Swan exponent of a local Galois representation.
The contents of this subsection are based on \cite[Chapter VI]{Ser79} and \cite[Chapter 19]{Ser77}.

Let $F$ be a non-Archimedean local field with residue characteristic $p>0$.
Let $E$ be a finite Galois extension of $F$.
Let $G:=\Gal(E/F)$ and $\{G_{i}\}_{\in\Z_{\geq0}}$ be the filtration by the lower ramification subgroups of $G$; recall that $\{G_{i}\}_{i\geq0}$ is decreasing and satisfies $G_{i}=\{1\}$ for sufficiently large $i$ (see \cite[Chapter IV]{Ser79}).
We write $g_{i}$ for the cardinality of $G_{i}$.

Suppose that $\rho$ is a representation of $G$ on a finite-dimensional $\C$-vector space $V$.
We simply call such $\rho$ a \textit{Galois representation (of $G$)}.
Then the \textit{Swan exponent} of $\rho$ is defined by
\[
\Sw(\rho)
:= \sum_{i\geq1} \frac{g_{i}}{g_{0}}\dim(V/V^{G_{i}}).
\]
Similarly, the \textit{Artin exponent} of $\rho$ is defined by
\[
\Art(\rho)
:= \sum_{i\geq0} \frac{g_{i}}{g_{0}}\dim(V/V^{G_{i}}).
\]

Let us review several basic properties of $\Sw(\rho)$ which will be useful to us.

\subsubsection{Swan exponent as a scalar product}\label{sssec:Swan}

The Swan exponent of $\rho$ can be also thought of as the scalar product of the character $\Tr(\rho)$ with the character $\mcS$ of the Swan representation ``$sw_{G}$'' (see \cite[Chapter 19]{Ser77}):
\[
\Sw(\rho)=\langle\Tr(\rho),\mcS\rangle.
\]

\subsubsection{Swan exponent via upper ramification filtration}\label{sssec:upper}

For a Galois representation $\rho$, we define the number $\slope(\rho)$ by
\[
\slope(\rho):=\inf\{s\in\R_{\geq0} \mid \text{$\rho$ is trivial on $G^{s}$}\},
\]
where $\{G^{r}\}_{r\geq0}$ is the filtration  of $G$ by the upper ramification subgroups (see \cite[Section IV.3]{Ser79} for the definition of the upper ramification subgroups).
When $\rho$ is irreducible, we have 
\[
\Sw(\rho)=\dim(\rho)\cdot\slope(\rho).
\]
Indeed, it is obvious when $\rho$ is trivial on $G_{0}=G^{0}$, so let us consider the case where $\rho$ is non-trivial on $G_{0}$.
As each $G_{i}$ is a normal subgroup of $G$, $V^{G_{i}}$ is a subspace of $V$ stable under the action of $G$.
Thus, by letting $m\in\Z_{\geq0}$ be the largest integer such that $G_{m}$ is non-trivial, the irreducibility of $\rho$ implies that
\[
V^{G_{i}}=
\begin{cases}
0&\text{if $0\leq i\leq m$},\\
V&\text{if $m< i$}.
\end{cases}
\]
Hence we have
\[
\Sw(\rho)
=
\dim(V)\sum_{i=1}^{m} \frac{g_{i}}{g_{0}}.
\]
The sum on the right-hand side is nothing but the value $\varphi_{E/F}(m)$ of the Herbrand function $\varphi_{E/F}$ with respect to $G=\Gal(E/F)$ at $m$ (see \cite[Section IV.3]{Ser79}).
By the definition of the upper ramification filtration, we have 
\[
G^{\varphi_{E/F}(m)}=G_{m}
\quad\text{and}\quad
G^{\varphi_{E/F}(m)+\varepsilon}=G_{m+1}
\]
for any sufficiently small positive number $\varepsilon$.
In other words, we have $\slope(\rho)=\varphi_{E/F}(m)=\Sw(\rho)/\dim(\rho)$.

Note that, by the definition of the slope, we have the following:
%The following is a special case of \cite[3.6. Corollaire 1]{Hen80}:
\begin{lem}\label{lem:Swan-twist}
Let $\rho_{1}$ and $\rho_{2}$ be irreducible Galois representations of $G$.
Then we have
%\[
%\Sw(\rho\otimes\omega)\leq \max\{\Sw(\rho), d\Sw(\omega)\},
%\]
\[
\slope(\rho_{1}\otimes\rho_{2})\leq \max\{\slope(\rho_{1}), \slope(\rho_{1})\},
\]
where the equality holds if $\slope(\rho_{1})\neq \slope(\rho_{2})$.
\end{lem}

\subsubsection{Behavior under a tamely ramified extension}\label{sssec:tame}

Let $F'$ be a tamely ramified extension of $F$ contained in $E$.
Note that this condition is equivalent to that the corresponding subgroup $G':=\Gal(E/F')$ of $G=\Gal(E/F)$ contains $G_{1}$.
Let $\{G'_{i}\}_{i\in\Z_{\geq0}}$ be the filtration of $G'$ by the lower ramification subgroups of $G'$.
By noting that $G'\supset G_{1}$ and $G'_{i}=G'\cap G_{i}$ for any $i\in\Z_{\geq0}$ (\cite[Chapter IV, Proposition 2]{Ser79}), we have $G'_{i}=G_{i}$ for any $i\in\Z_{\geq1}$.
This implies that
\[
\Sw(\rho|_{G'})
=
\frac{g_{0}}{g'_{0}}\cdot\Sw(\rho),
\]
where $g'_{0}$ denotes the cardinality of $G'_{0}$.
In particular, by letting $F'$ be the maximal tamely ramified extension of $F$ in $E$ so that $G'=G_{1}$, we have
\[
\Sw(\rho|_{G_{1}})
=\frac{g_{0}}{g_{1}}\cdot\Sw(\rho)
=e_{F'/F}\cdot\Sw(\rho),
\]
where $e_{F'/F}$ denotes the ramification index of $F'/F$.

\subsubsection{Behavior under the induction}\label{sssec:Induction}

Let $F'$ be an extension of $F$ contained in $E$.
For any finite-dimensional representation $\sigma$ of the Galois group $\Gal(E/F')$, we have
\[
\Sw(\Ind_{F'/F}\sigma)
=
f_{F'/F}\cdot\Sw(\sigma)+(v_{F}(\mfd_{F'/F})-[F':F]+f_{F'/F})\cdot\dim(\sigma),
\]
where $f_{F'/F}$ denotes the residue degree of $F'/F$ and $\mfd_{F'/F}$ is the discriminant of $F'/F$ (see \cite[Corollary of Proposition 4, Chapter VI]{Ser79}, in which the formula is stated in terms of Artin exponents).

In particular, we have the following:
\begin{lem}\label{lem:Ind-Swan}
Suppose that $F'/F$ is a ramified quadratic extension of dyadic fields.
Then we have
\[
\Sw(\Ind_{F'/F}\sigma)
=
\Sw(\sigma)+\Sw(\omega_{F'/F})\cdot\dim(\sigma),
\]
where $\omega_{F'/F}$ denotes the quadratic character of $F^{\times}$ corresponding to $F'/F$.
\end{lem}

\begin{proof}
By applying the above formula, we have
\[
\Sw(\Ind_{F'/F}\sigma)
=
\Sw(\sigma)+(v_{F}(\mfd_{F'/F})-1)\cdot\dim(\sigma).
\]
We have $v_{F}(\mfd_{F'/F})=\Sw(\omega_{F'/F})+1$ by \cite[Corollary 2 of Proposition 6, Chapter VI]{Ser79} (note that the symbol $f(-)$ in \textit{loc.\ cit.}\ denotes the Artin exponent), hence we get the assertion.
\end{proof}

\subsection{Difference of $\Sym^{2}$ and $\wedge^{2}$ Swan exponents: the case where $p\neq2$}

In the rest of this section, we investigate the difference of the Swan exponents of the symmetric and exterior squares of a Galois representation.
We first consider the case where the residue characteristic $p$ of $F$ is odd.

Let us start with the following observation.
For any class function $f\colon G\rightarrow \C$, we define $\Psi^{2}f\colon G\rightarrow\C$ by $\Psi^{2}f(g):=f(g^{2})$.
Note that then $\Psi^{2}f$ is again a class function on $G$.
When $\rho$ is a finite-dimensional representation of $G$ whose dimension is $d$, for $g\in G$, we have
\[
\Tr(\Sym^{2}\rho)(g)-\Tr(\wedge^{2}\rho)(g)
=\sum_{i=1}^{d}\alpha_{i}^{2}
=\Tr(\rho)(g^{2})
=\Psi^{2}\Tr(\rho)(g),
\]
where $\alpha_{1},\ldots,\alpha_{d}\in \C^{\times}$ are the eigenvalues of $\rho(g)$.
In other words, $\Psi^{2}\Tr(\rho)$ is the character of the virtual representation  $\Sym^{2}\rho-\wedge^{2}\rho$ of $G$.
Let us write $\Psi^{2}\rho:=\Sym^{2}\rho-\wedge^{2}\rho$.

\begin{prop}\label{prop:p-odd}
Suppose that $p\neq2$.
For any Galois representation $\rho$, we have
\[
\Sw(\Sym^{2}\rho)-\Sw(\wedge^{2}\rho)
=
\Sw(\rho).
\]
\end{prop}

\begin{proof}
If we replace $F$ with the maximal tamely ramified extension $F'$ of $F$ in $E$,  then each term in the equality in the assertion is multiplied by $e_{F'/F}$ (see Section \ref{sssec:tame}).
Thus, we may assume that $G=G_{1}$ from the beginning.
Note that then $G$ is a finite $p$-group.
If we let $p^{r}$ be the order of $G$, then $\rho$ is defined over $\Q(\mu_{p^{r}})$, where $\mu_{p^{r}}$ denotes the subset of $\C$ consisting of $p^{r}$-th roots of unity (see \cite[Proposition 33]{Ser77}).

Let $g\in G$ and $\alpha_{1},\ldots,\alpha_{d}\in \C^{\times}$ be the eigenvalues of $\rho(g)$, where $d:=\dim(\rho)$.
Let $\tau\in\Gal(\Q(\mu_{p^{r}})/\Q)\cong (\Z/p^{r}\Z)^{\times}$ be the element corresponding to $2\in (\Z/p^{r}\Z)^{\times}$, i.e., $\tau(\alpha)=\alpha^{2}$ for any $\alpha\in \mu_{p^{r}}$.
As every eigenvalue $\alpha_{i}$ belongs to $\mu_{p^{r}}$, the class function $\Psi^{2}\Tr(\rho)$ is equal to the character $\Tr({}^{\tau}\!\rho)$ of the $\tau$-twist of $\rho$.
Hence we have $\Psi^{2}\rho={}^{\tau}\!\rho$.
Since the Galois twist does not change the Swan exponent, we get
$\Sw(\Sym^{2}\rho)-\Sw(\wedge^{2}\rho)=\Sw(\Psi^{2}\rho)=\Sw({}^{\tau}\!\rho)=\Sw(\rho)$.
\end{proof}

By considering the last non-trivial lower ramification subgroup, we can also prove the above theorem in the following way:
\begin{proof}[Another proof of Proposition \ref{prop:p-odd}]
As in the above proof, let us assume that $G=G_{1}$ and $\rho$ is a non-trivial irreducible representation of $G$.
Let $G_{m}$ be the last non-trivial lower ramification group of $G$, i.e., $G_{m+1}=\{1\}$.
Then $G_{m}$ is central in $G$ by \cite[Chapter IV, Proposition 10]{Ser79}.
Hence $G_{m}$ acts on $\rho$ via a character $\chi$.
Then $G_{m}$ acts on both $\Sym^{2}\rho$ and $\wedge^{2}\rho$ via $\chi^{2}$.
As $G_{m}$ is a $p$-group, where $p$ is odd, $\chi^{2}$ is non-trivial.
This means that any irreducible constituent of $\Sym^{2}\rho$ and $\wedge^{2}\rho$ has the same slope as $\rho$.
Thus, by noting that $\dim(\Sym^{2}\rho)-\dim(\wedge^{2}\rho)=\dim(\rho)$, we get $\Sw(\Sym^{2}\rho)-\Sw(\wedge^{2}\rho)=\Sw(\rho)$ (see Section \ref{sssec:upper}).
\end{proof}

\begin{cor}\label{cor:p-odd}
Suppose that $p\neq2$.
Let $\rho$ be a $2n$-dimensional irreducible self-dual Galois representation such that $(\Sw(\rho),2n)=1$.
Then we have
\[
\Sw(\Sym^{2}\rho)=n\cdot\Sw(\rho)
\quad\text{and}\quad
\Sw(\wedge^{2}\rho)=(n-1)\cdot\Sw(\rho).
\]
\end{cor}

\begin{proof}
As $\rho$ is self-dual, we have $\rho\otimes\rho^{\vee}\cong\rho\otimes\rho\cong\Sym^{2}\rho\oplus\wedge^{2}\rho$.
By the explicit formula of Bushnell--Henniart--Kutzko (\cite[6.5, Theorem]{BHK98-JAMS}, see also \cite[3.5]{BH19}), the assumption that $(\Sw(\rho),2n)=1$ implies that
\[
\Sw(\rho\otimes\rho^{\vee})
=
(2n-1)\Sw(\rho).
\]
Thus, by combining this equality with Proposition \ref{prop:p-odd}, we get the assertion.
\end{proof}

\begin{rem}\label{rem:AM16}
\begin{enumerate}
\item
It is known that there is no irreducible self-dual Galois representation whose dimension is odd and greater than $1$ if the residue characteristic of $F$ is odd (see, e.g., \cite[84page, Proposition 4]{Pra99}).
Hence the assumption that the dimension of $\rho$ is even in Corollary \ref{cor:p-odd} automatically follows from the irreducibility and self-duality assumptions.
\item
Proposition \ref{prop:p-odd} is not new; it has been already obtained in the work of Anandavardhanan--Mondal (\cite[Theorem 1.2]{AM16}).
Thus the point here is that our proof is different from theirs.
\item
Since the explicit formula of Bushnell--Henniart--Kutzko works for general irreducible $\rho$, it is possible to get an explicit formula of $\Sw(\Sym^{2}\rho)$ and $\Sw(\wedge^{2}\rho)$ for any irreducible $\rho$ as in the manner of the proof of Corollary \ref{cor:p-odd}.
(We put various assumptions on $\rho$ just because then the formula is simplified.)
In fact, this is exactly how Anandavardhanan--Mondal established an explicit formula of $\Sw(\Sym^{2}\rho)$ and $\Sw(\wedge^{2}\rho)$.
\end{enumerate}
\end{rem}

\subsection{Difference of $\Sym^{2}$ and $\wedge^{2}$ Swan exponents: the case where $p=2$}

We next consider the case where the residue characteristic of $F$ is $p=2$.

\begin{thm}\label{thm:p-2-weak-ineq}
Suppose that $p=2$.
For any Galois representation $\rho$ of $G$, we have
\begin{align}\label{ineq:weak}
0
\leq
\Sw(\Sym^{2}\rho)-\Sw(\wedge^{2}\rho)
\leq
2\Sw(\rho).
\end{align}
\end{thm}

\begin{proof}
By replacing $G$ with $G_{1}$, we may assume that $G=G_{1}$ (the same argument as in the proof of the case where $p\neq2$).
We first note that the quantity $\Sw(\Sym^{2}\rho)-\Sw(\wedge^{2}\rho)$ is given by applying the operation $\Psi^{2}$ to $\Tr(\rho)$ and then taking the scalar product with $\mcS$ (see Section \ref{sssec:Swan}).
In particular, $\Sw(\Sym^{2}\rho)-\Sw(\wedge^{2}\rho)$ is additive in $\rho$.
Thus it is enough to prove the inequalities when $\rho$ is irreducible. 
Since the assertion is obvious when $\rho$ is trivial, we may also assume that $\rho$ is non-trivial.

We put $d:=\dim(\rho)$.
We note that $\dim(\Sym^{2}\rho)-\dim(\wedge^{2}\rho)=\dim(\rho)$.
By the definition of the Swan exponent, it suffices to show that, for any $i>0$,
\[
0
\leq  d-\dim((\Sym^{2}\rho)^{G_i})+\dim((\wedge^{2}\rho)^{G_i})
\leq 2d-2\dim(\rho^{G_i}),
\]
or equivalently
\begin{align}\label{ineq:p=2-proof}
-d
\leq \dim((\wedge^{2}\rho)^{G_{i}})-\dim((\Sym^{2}\rho)^{G_i})
\leq d-2\dim(\rho^{G_i}).
\end{align}
By noting that this is additive in $\rho$ and depends only on $G_{i}$, we may assume that $G=G_{i}$ and $\rho$ is irreducible non-trivial by the same argument as above.
Then we have $\rho^{G}=0$ and
\[
\dim((\Sym^{2}\rho)^G)+\dim((\wedge^{2}\rho)^G)
=\dim((\rho\otimes\rho)^G)
=\begin{cases}
1 & \text{if $\rho$ is self-dual,}\\
0 & \text{otherwise.}
\end{cases}
\]
Hence only $0,\pm1$ are the possibilities of the value of $\dim((\wedge^{2}\rho)^G)-\dim((\Sym^{2}\rho)^G)$, which gives the inequalities \eqref{ineq:p=2-proof}.
\end{proof}

Let us now discuss the cases where the equalities of \eqref{ineq:weak} are attained.
Note that all terms in \eqref{ineq:weak} are $0$ when $\rho$ is tame, i.e., trivial on $G_{1}$. 

\begin{prop}\label{prop:weak-strict}
Suppose that $p=2$.
Let $\rho$ be a Galois representation of $G$ which is not tame.
Then we have
\[
0
\leq
\Sw(\Sym^{2}\rho)-\Sw(\wedge^{2}\rho)
<
2\Sw(\rho).
\]
The equality $\Sw(\Sym^{2}\rho)-\Sw(\wedge^{2}\rho)=0$ holds if and only if the restriction of $\rho$ on $G_{1}$ decomposes into the sum of quadratic characters.
\end{prop}

\begin{proof}
By the same argument as in the proof of Theorem \ref{thm:p-2-weak-ineq}, it suffices to show the inequalities 
\begin{align}\label{ineq:weak-strict}
-d\leq\dim((\wedge^{2}\rho)^{G})-\dim((\Sym^{2}\rho)^{G}) < d
\end{align}
and discuss the condition so that the left equality holds by assuming $G=G_{1}$ and $\rho$ is a non-trivial irreducible representation of $G$ whose dimension is $d$.
The proof of Theorem \ref{thm:p-2-weak-ineq} shows that the inequalities of \eqref{ineq:weak-strict} always hold and the equality on the left never holds when $d>1$.
When $d=1$, i.e., $\rho$ is a (non-trivial) character, we have $\dim((\wedge^{2}\rho)^G)=0$ and 
\[
\dim((\Sym^{2}\rho)^G)
=
\begin{cases}
0 & \text{if $\rho^{2}$ is non-trivial,}\\
1 & \text{if $\rho^{2}$ is trivial.}
\end{cases}
\]
This completes the proof.
\end{proof}

As an application of the above proposition, we have the following.

\begin{prop}\label{prop:ssc-p-2}
Suppose that $p=2$.
Let $\rho$ be a $(2n+1)$-dimensional irreducible self-dual Galois representation such that $(\Sw(\rho),2n+1)=1$.
Then we have
\[
\Sw(\Sym^{2}\rho)-\Sw(\wedge^{2}\rho)=0.
\]
\end{prop}

\begin{proof}
Since $p$ is prime to $\dim(\rho)=2n+1$, the classification result of self-dual irreducible Galois representations due to Bushnell--Henniart implies that $\rho$ is of the form $\Ind_{L/F}(\chi)$, where $L/F$ is a tamely ramified extension of degree $2n+1$ and $\chi$ is a quadratic character of $L^{\times}$ (\cite[Section 3]{BH11}).
In other words, there exist a subgroup $H$ of $G$ containing $G_{1}$ and a quadratic character $\chi$ of $H$ such that $\rho\cong\Ind_{H}^{G}(\chi)$.
Thus we have $\rho|_{G_{1}}\cong\bigoplus_{g\in G/H}{}^{g}\chi|_{G_{1}}$. Hence the assumption in the last sentence of Proposition \ref{prop:weak-strict} is satisfied and we get the assertion.
\end{proof}

\begin{rem}
Note that the assumption that $\dim(\rho)$ is odd is not automatic when $p=2$ and indispensable in the above proof so that the result of Bushnell--Henniart \cite{BH11} is available.
\end{rem}

By the same argument as in the case where $p\neq2$ (i.e., using the Bushnell--Henniart--Kutzko formula), we obtain the following:

\begin{cor}\label{cor:ssc-p-2}
With the same assumptions as in Proposition \ref{prop:ssc-p-2}, we have
\[
\Sw(\Sym^{2}\rho)
=\Sw(\wedge^{2}\rho)
=n\cdot\Sw(\rho).
\]
\end{cor}

\section{A conjecture and some evidence}
Let us keep the notation as in the previous section.

\begin{conj}\label{conj:Henniart}
Assume $p=2$.
For any Galois representation $\rho$, we have
\begin{align}\label{ineq:Henniart}
\Sw(\Sym^{2}\rho)-\Sw(\wedge^{2}\rho)
\leq
\Sw(\rho).
\end{align}
\end{conj}

We remark that, by the same argument as in the proof of Theorem \ref{thm:p-2-weak-ineq}, the above conjecture is equivalent to the following:

\begin{conj}\label{conj:Henniart2}
Assume $p=2$ and $G=G_{1}$.
For any irreducible faithful Galois representation $\rho$, we have
\[
\Sw(\Sym^{2}\rho)-\Sw(\wedge^{2}\rho)
\leq
\Sw(\rho).
\]
\end{conj}

In the following, let us suppose that $p=2$ and $G=G_{1}$.
In particular, $\dim(\rho)$ is a power of $2$.
Conjecture \ref{conj:Henniart2} is obvious when $\dim(\rho)=1$, so let us investigate the case where $\dim(\rho)>1$ in the following.
Note that then $G$ is necessarily non-abelian since any irreducible representation of an abelian group is $1$-dimensional.

\begin{rem}
Note that the proof of Theorem \ref{thm:p-2-weak-ineq} used only a descending filtration by subgroups, but no particular property of that filtration. 
The stronger conjecture (Conjecture \ref{conj:Henniart}), if true, necessarily uses strong properties of the filtration by ramification subgroups, as can be seen in the examples in this section.
\end{rem}

\subsection{The $2$-dimensional case}\label{subsec:2-dim}
%\MO{The content of this section is based on Guy's e-mail on 20220520.}
We first investigate the case where $\dim(\rho)=2$.
Since any irreducible representation of a $p$-group is monomial, i.e., induced from a $1$-dimensional representation of a subgroup (\cite[Theorems 14 and 16]{Ser77}), we may write $\rho\cong\Ind_{G'}^{G}(\chi)$, where $G'=\Gal(E/F')$ for a quadratic separable (ramified) extension $F'$ of $F$ and $\chi$ is a character of $G'$.
We also see $\chi$ as a character of $F^{\prime\times}$ by class field theory.
Let $\omega:=\omega_{F'/F}$ be the quadratic character of $F^{\times}$ defining $F'/F$.
Let $\chi^{\circ}:=\chi|_{F^{\times}}$.
We put $a:=\Sw(\chi)$, $b:=\Sw(\chi^{2})$, $c:=\Sw(\chi^{\circ})$, $d:=\Sw(\omega\chi^{\circ})$, and $s:=\Sw(\omega)$.

\begin{lem}\label{lem:relations}
We have the following relations:
\begin{enumerate}
\item $d\leq\max\{c,s\}$, where equality holds when $c\neq s$,
\item $a\geq2c$,
\item $0<s\leq2e$, where $e:=\val_F(2)$,
\item $a\geq s$.
\end{enumerate}
\end{lem}

\begin{proof}
The inequality (1) is obvious.

Since $F'/F$ is a ramified quadratic extension, we get the inequality (2).

As $F'/F$ is wildly ramified, we have $s>0$.
On the other hand, as used in the proof of Lemma \ref{lem:Ind-Swan}, we have $v_{F}(\mfd_{F'/F})=s+1$.
By the upper bound formula for $v_{F}(\mfd_{F'/F})$ (see \cite[Chapter III, Section 6, Remarks (1)]{Ser79}), we get $v_{F}(\mfd_{F'/F})\leq 1+2e$.
Thus we obtain $s\leq2e$, hence the inequalities (3).

Since $\rho\cong\Ind_{G'}^{G}(\chi)$ is irreducible, we necessarily have $\chi^{\sigma}\neq\chi$, where $\sigma$ is the non-trivial element of $\Gal(F'/F)$.
If $a<s$, then $\sigma$ acts trivially on $F^{\prime\times}/1+\mfp_{F'}^{a+1}$, which implies that $\chi^{\sigma}\cdot\chi^{-1}$ is trivial, hence a contradiction.
Thus we get the inequality (4).
\end{proof}

\begin{lem}\label{lem:squaring}
We have $b\leq\max\{a/2,a-2e\} \,(<a)$.
More precisely, we have
\begin{itemize}
\item
$b=a-2e$ when $a>4e$,
\item
$b=a/2$ when $a<4e$ and $a$ is even,
\item
$b\leq a/2$ otherwise.
\end{itemize}
\end{lem}

\begin{proof}
The first assertion simply follows from that $(1+\mfp_{F'}^{b})^{2}\subset 1+\mfp_{F'}^{\min\{b+2e,2b\}}$.

To show the latter assertion, let us examine the behavior of the squaring map on $\mcO_{F'}^{\times}$.
We have $(1+x)^{2}=1+2x+x^{2}$ for any $x\in F'$.
If $v_{F'}(x)=i\geq1$, then $v_{F'}(x^{2})=2i$ and $v_{F'}(2x)=i+2e$ (recall that $e=v_{F}(2)$).
If follows that $v_{F'}(2x+x^{2})\geq\min\{2i,i+2e\}$ with equality when $2i\neq i+2e$.
Recall that, by definition, $a\in\Z_{>0}$ is the integer such that $\chi$ is trivial on $1+\mfp_{F'}^{a+1}$, but not on $1+\mfp_{F'}^{a}$.
(Note that we have $a>0$ by Lemma \ref{lem:relations} (3) and (4).)

Let us first assume $a>4e$ and take $i=a-2e$.
Then we have $2i>i+2e$, hence $\min\{2i,i+2e\}=a$.
Moreover, we have $(1+x)^{2}=1+2x \pmod{\mfp_{F'}^{a+1}}$.
Consequently, there exists an element $x\in F'$ with valuation $i$ such that $\chi(1+x)^{2}\neq1$, thus $b\,(:=\Sw(\chi^{2}))$ is at least $a-2e$. 
On the other hand, if we take $i=a-2e+1$, then we have $\min\{2i,i+2e\}\geq a+1$.
Hence we have $\chi(1+x)^{2}=1$ for any element $x\in F'$ with valuation $i$, which implies that $b=a-2e$.

Let us then assume $a<4e$ and $a$ is even. 
We take $i=a/2$. 
Then we have $2i<i+2e$ and $\min\{2i,i+2e\}=2i=a$.
Hence we see that $(1+x)^{2}=1+x^{2} \pmod{\mfp_{F'}^{a+1}}$.
By noting that the squaring map is bijective on $k_{F'}^{\times}$, we can choose $x\in F'$ with valuation $i$ satisfying $\chi(1+x)^{2}\neq1$.
Similarly with $i=a/2+1$, we see that $\chi(1+x)^{2}=1$ for any element $x\in F'$ with valuation $i$.
Thus we get $b=a/2$.

When $a<4e$ and $a$ is odd or $a=4e$, we can only assert that $b\leq a/2$.
\end{proof}

By Lemma \ref{lem:Ind-Swan}, we have $\Sw(\rho)=a+s$.
Note that we have
\[
\wedge^{2}(\Ind_{F'/F}(\chi))\cong\omega\chi^{\circ}
\quad
\text{and}
\quad
\Sym^{2}(\Ind_{F'/F}(\chi))\cong\Ind_{F'/F}(\chi^{2})\oplus\chi^{\circ}
\]
(see Proposition \ref{prop:Ind-Sym-Ext}).
Hence we get
\begin{align*}
\Sw(\Sym^{2}\rho)-\Sw(\wedge^{2}\rho)
&=\Sw(\Ind_{F'/F}(\chi^{2}))+\Sw(\chi^{\circ})-\Sw(\omega\chi^{\circ})\\
&=b+s+c-d,
\end{align*}
where we again used Lemma \ref{lem:Ind-Swan} in the second equality.
We investigate the conjectural inequality \eqref{ineq:Henniart} by a case-by-case argument:
\begin{enumerate}
\item
If $c>s$, then $d=c$ by Lemma \ref{lem:relations} (1).
Thus, by using Lemma \ref{lem:squaring}, we have 
\[
\Sw(\Sym^{2}\rho)-\Sw(\wedge^{2}\rho)=b+s+c-c=b+s< a+s=\Sw(\rho).
\]
Note that this is the strict inequality.
\item
If $c<s$, then $d=s$ by Lemma \ref{lem:relations} (1).
Thus, by using Lemma \ref{lem:squaring}, we have 
\[
\Sw(\Sym^{2}\rho)-\Sw(\wedge^{2}\rho)=b+s+c-s=b+c<a+s=\Sw(\rho).
\]
Note that this is the strict inequality.
\item
If $c=s$, we have $d\leq c$ by Lemma \ref{lem:relations} (1).
Let us show that the weak inequality 
\[
\Sw(\Sym^{2}\rho)-\Sw(\wedge^{2}\rho)=b+s+c-d\leq a+s=\Sw(\rho),
\]
or equivalently, $a-b-c+d\geq0$ holds in this case.
\begin{enumerate}
\item
Suppose that $a>4e$, hence $b=a-2e$ by Lemma \ref{lem:squaring}.
Since we have $s\leq2e$ by Lemma \ref{lem:relations} (3), we get $b+c=a-2e+s\leq a$, which implies that $a-b-c+d\geq0$.
Note that we have the strict inequality unless $d=0$ and $s=2e$.
\item
Suppose that $a<4e$ and $a$ is even, hence $b=a/2$ by Lemma \ref{lem:squaring}.
Since we have $a\geq 2c$ by Lemma \ref{lem:relations} (2), we get $b+c\leq a$, which implies that $a-b-c+d\geq0$.
Note that we have the strict inequality unless $d=0$ and $a=2c$.
\item
Suppose that $a\leq4e$, hence $b\leq a/2$ by Lemma \ref{lem:squaring}.
Thus the same discussion as in the previous case works.
Note that we have the strict inequality unless $d=0$ and $a=2b=2c$.
\end{enumerate}
Note that $d=0$ means that $\rho$ has trivial determinant (recall that we are assuming $G=G_{1}$), i.e. is symplectic.
\end{enumerate}

Thus we obtain the following:
\begin{thm}\label{thm:Henniart2-dim2}
Conjecture \ref{conj:Henniart2} is true if $\dim(\rho)=2$.
Moreover, the inequality is strict unless $\rho$ is symplectic.
\end{thm}

\begin{rem}
Let us present an example where the equality $\Sw(\Sym^{2}\rho)-\Sw(\wedge^{2}\rho)=\Sw(\rho)$ holds.
Take $F'/F$ to have $s=1$, and take $\chi$ such that $a=2$ and $\chi^{\circ}=\omega$; then we have $c=1, b=1, d=0$.
\end{rem}

\subsection{The case where $\rho$ is induced along a cyclic extension}\label{subsec:cyclic}

\begin{prop}
Suppose that $\rho$ is an irreducible Galois representation of the form $\rho=\Ind_{F'/F}(\chi)$, where $F'/F$ is cyclic and $\chi$ is a character of $F^{\prime\times}$. 
Then the inequality \eqref{ineq:Henniart} holds.
\end{prop}

\begin{proof}
Let $\rho=\Ind_{F'/F}(\chi)$, where $F'/F$ is cyclic and $\chi$ is a character of $F^{\prime\times}$. 
In other words, $F'$ is a subfield of $E$ corresponding to a normal subgroup $G'$ of $G$ such that $G/G'\cong\Gal(F'/F)$ is cyclic, say order $2n$.
Let $c$ be a generator of $G/G'$.
Let $H$ be a subgroup of $G$ which is the preimage of $\langle c^{n}\rangle\subset G/G'$.
Let $L$ be the subfield of $F'$ corresponding to $H$.
%Recall that, as we are assuming that $G=G_{1}$, the dimension of $\rho$ is even ($2$-power), say $2n$.

By the Frobenius formula for induced representations, we have
\[
\Tr(\Psi^{2}\rho)(g)
=
\Tr(\rho)(g^{2})
=
\begin{cases}
\sum_{\tau\in G/G'} \chi^{\tau}(g^{2}) & \text{if $g^{2}\in G'$,}\\
0 & \text{otherwise.}
\end{cases}
\]
Note that, as $G/G'$ is cyclic, we have $g^{2}\in G'$ if and only if $g\in H$.
Similarly, we have
\[
\Tr(\Ind_{F'/F}(\chi^{2}))(g)
=
\begin{cases}
\sum_{\tau\in G/G'} (\chi^{2})^{\tau}(g) & \text{if $g\in G'$,}\\
0 & \text{otherwise.}
\end{cases}
\]
We let $\chi^{\circ}$ denote the character of $H$ corresponding to $\chi|_{L^{\times}}$ under local class field theory.
By letting $\omega$ be the non-trivial character of $H/G'$, we have
\[
\Tr(\Ind_{L/F}(\chi^{\circ})-\Ind_{L/F}(\chi^{\circ}\omega))(g)
=
\begin{cases}
0 & \text{if $g\in G'$,}\\
2\sum_{\tau\in G/H} (\chi^{\circ})^{\tau}(g) & \text{if $g\in H\smallsetminus G'$,}\\
0 & \text{otherwise.}
\end{cases}
\]

Therefore, by noting that $\chi(g^{2})=\chi^{\circ}(g)$ for any $g\in H\smallsetminus G'$ (see \cite[29.1, Transfer theorem]{BH06}), we get
\[
\Tr(\Psi^{2}\rho)
=
\Tr(\Ind_{F'/F}(\chi^{2}) + \Ind_{L/F}(\chi^{\circ}) - \Ind_{L/F}(\chi^{\circ}\omega)).
\]
Hence
\begin{align*}
\Sym^{2}\rho-\wedge^{2}\rho
&\cong
\Ind_{F'/F}(\chi^{2}) + \Ind_{L/F}(\chi^{\circ}) - \Ind_{L/F}(\chi^{\circ}\omega)\\
&\cong
\Ind_{L/F}\bigl(\Ind_{F'/L}(\chi^{2}) + \chi^{\circ}- \chi^{\circ}\omega\bigr).
\end{align*}
On the other hand, we also have $\rho\cong \Ind_{L/F}(\Ind_{F'/L}(\chi))$.
Hence, by Lemma \ref{lem:Ind-Swan}, it suffices to show the inequality 
\begin{align}\label{ineq:cyclic}
\Sw(\Ind_{F'/L}(\chi^{2})) + \Sw(\chi^{\circ}) - \Sw(\chi^{\circ}\omega)
\leq
\Sw(\Ind_{F'/L}(\chi))
\end{align}
(note that $\dim(\Sym^{2}\rho)-\dim(\wedge^{2}\rho)=\dim(\rho)$).

As observed in the proof of Theorem \ref{thm:Henniart2-dim2}, this inequality \eqref{ineq:cyclic} is nothing but the inequality \eqref{conj:Henniart} for $\Ind_{F'/L}(\chi)$, which we have already proved.
\end{proof}

\subsection{Approach based on the structure of ramification subgroups}\label{subsec:approach-ramif-group}

We next investigate the conjectural inequality by looking at the last two steps of the ramification filtration.
Let us assume that $G=G_{1}\neq\{1\}$ and $\rho$ is an irreducible faithful representation of $G$.
We let $n=\dim(\rho)$, a power of $2$. 
Since the cases where $n=1,2$ are already treated, let us assume that $n>2$ (most considerations also work when $n=2$).
Note that the center $Z(G)$ of $G$ is cyclic by Lemma \ref{lem:p-group-faithful-center}.

Let $Z\subset G$ be the last non-trivial lower ramification subgroup.
Since $G=G_1$, $Z$ is contained in the center $Z(G)$ of $G$ (this follows from \cite[Chapter IV, Proposition 10]{Ser79}).
By noting that $Z$ has an $\F_{2}$-vector space structure (\cite[Chapter IV, Section 2, Corollary 3]{Ser79}) and that $Z(G)$ is cyclic, the order of $Z$ is $2$ and $Z\neq G$ since $n>1$.

Let us look at the smallest ramification subgroup $H$ containing $Z$ properly.
The quotient $V=H/Z$ is a non-trivial $\F_{2}$-vector space, equipped with a quadratic form $q\colon V \rightarrow Z\cong\F_{2};\, v\mapsto v^{2}$ and an alternating bilinear form $b\colon V\times V \rightarrow Z\cong\F_{2};\, (v_{1},v_{2})\mapsto v_{1}v_{2}v_{1}^{-1}v_{2}^{-1}$.
(See \cite[Chapter IV, Proposition 10]{Ser79}.)
The center $Y$ of $H$ is the inverse image of 
\[
V^{0}:=\{v_{0}\in V \mid \text{$b(v_{0},v)=0$ for any $v\in V$}\}
\]
in $H$.
Note that $q|_{V^{0}}\colon V^{0}\rightarrow Z$ is a group homomorphism.
We have either that $q|_{V^{0}}$ is trivial (then $Y$ has exponent $2$) or that $q|_{V^{0}}$ is non-trivial, hence $\Ker(q|_{V^{0}})$ is a hyperplane in $V^{0}$ (then $Y$ has elements of order $4$).

The group $G$ acts by conjugation on $H$ and $Y$.
As we have $G=G_{1}$, $G$ acts trivially on $V=H/Z$ (again by \cite[Chapter IV, Proposition 10]{Ser79}); this action preserves $b$ and $q$. 
Similarly, the action of $G$ on $V^{0}=Y/Z$ is trivial, hence any $g\in G$ defines a homomorphism $\phi_{g}\colon V^{0}\rightarrow Z$ given by $v\mapsto gvg^{-1}v^{-1}$.
Since $\phi_{g}$ is multiplicative on $g$, we get a homomorphism
\[
\phi\colon G\rightarrow\Hom_{\F_{2}}(V^{0},Z);\quad g\mapsto\phi_{g},
\]
which is trivial on $H$.
We define a subspace $W$ of $V^{0}$ by $W:=\bigcap_{g\in G}\Ker\phi_{g}$.
If we let $I\subset Y$ be the inverse image of $W$ in $Y$, then we have $I=H\cap Z(G)$.
As $Z(G)$ is cyclic, we see that $I$ is either $Z$ or cyclic of order $4$.

\[
\xymatrix@R=10pt{
H\ar@{->>}[r] & V:=H/Z \ar^-{q}[r]& Z\\
Y\ar@{->>}[r]\ar@{}[u]|{\cup} & V^{0}\ar_-{\quad q|_{V^{0}}}[ru]\ar@{}[u]|{\cup} & \\
I \ar@{}[u]|{\cup}\ar@{->>}[r]&W \ar@{}[u]|{\cup}&\\
Z \ar@{}[u]|{\cup}\ar@{->>}[r]&0 \ar@{}[u]|{\cup}&\\
}
\]

\subsubsection{}
Let us first consider the case where $I$ is cyclic of order $4$.

\begin{prop}\label{prop:I-order-4}
When $I$ is of order $4$, the strict inequality holds in Conjecture \ref{conj:Henniart2}.
\end{prop}

\begin{proof}
We write $Z=G^{\alpha}$ and $H=G^{\alpha'}$ via the upper numbering.
As $\rho$ is irreducible and faithful, we have $\slope(\rho)=\alpha$, hence $\Sw(\rho)=n\alpha$ (see Section \ref{sssec:upper}).
Since $I$ acts on $\rho$ as a faithful $1$-dimensional character, any generator of $I$ acts by $-1$ on $\rho\otimes\rho$. 
This implies that any irreducible constituent of $\rho\otimes\rho\cong\Sym^{2}\rho\oplus\wedge^{2}\rho$ has slope $\alpha'$.
Hence we have $\Sw(\Sym^2\rho)=n(n+1)\alpha'/2$ and $\Sw(\wedge^2\rho)=n(n-1)\alpha'/2$.
Thus we get 
\[
\Sw(\Sym^2\rho)-\Sw(\wedge^2\rho)
=n\alpha'
<n\alpha
=\Sw(\rho),
\]
which is the strict inequality.
\end{proof}

\subsubsection{}
We now assume that $I=Z$, which means that $W=0$.
Note that $\mathrm{Im}(\phi)$ is a subspace of the dual $(V^{0})^{\vee}$ of $V^{0}$, and its orthogonal in $V^{0}$ is $W$ by the definition of $W$.
Hence we have $\mathrm{Im}(\phi)=\Hom_{\F_{2}}(V^0,Z)$.
Now an irreducible representation of $H$ non-trivial on $Z$ is determined by its central character $\chi$, which can be any character on $Y$ which is non-trivial on $Z$ (see Section \ref{subsec:Heisenberg}).
The action of $g$ in $G$ on $\chi$ is obtained by multiplying $\chi$ by $\phi_g$.
An irreducible component of the restriction of $\rho$ to $H$ is thus labelled by such a character $\chi$, and $G$ acts transitively on (the classes of) such irreducible components.

\begin{prop}\label{prop:I-equal-Z-and-H-irr}
When $I=Z$ and $\rho$ restricts irreducibly to $H$, the inequality holds.
Moreover, the inequality is strict unless $\rho$ is symplectic.
\end{prop}

In the rest of this subsection, we prove this proposition.
Since $\rho|_{H}$ is irreducible, the center $Y$ of $H$ is equal to $I=Z$ by the same reasoning as in the beginning of this subsection.
Hence we have $V^0=0$. 
In particular $V$ is a symplectic vector space (meaning that the alternating form on it is non-degenerate). 
That representation is self-dual because the central character is the order $2$ character of $Z$.

We write $Z=G^\alpha$ and $H=G^{\alpha'}$ as above (thus $\alpha'<\alpha$).
Then we have $\Sw(\rho)=n\alpha$.
Since $\rho|_{H}$ is irreducible and self-dual, $\rho\otimes\rho|_{H}$ contains the trivial representation exactly once.
This implies that exactly one irreducible constituent of $\rho\otimes\rho$ has slope smaller than $\alpha'$.
Moreover, such a constituent must be a $1$-dimensional character.

\begin{enumerate}
\item
If $\rho$ is self-dual, then $\rho\otimes\rho$ contains the trivial representation with multiplicity one. 
\begin{enumerate}
\item
If $\rho$ is orthogonal (i.e., the trivial representation appears in $\Sym^{2}\rho$), we have $\Sw(\Sym^2\rho)=\alpha'(n(n+1)/2-1)$ and $\Sw(\wedge^2\rho)=\alpha' n(n-1)/2$, hence 
\[
\Sw(\Sym^2\rho)-\Sw(\wedge^2\rho)
=(n-1)\alpha'
<n\alpha
=\Sw(\rho).
\]
\item
If $\rho$ is symplectic (i.e., the trivial representation appears in $\wedge^{2}\rho$), we have $\Sw(\Sym^2\rho)=\alpha' n(n+1)/2$ and $\Sw(\wedge^2\rho)=\alpha'(n(n-1)/2-1)$, hence 
\[
\Sw(\Sym^2\rho)-\Sw(\wedge^2\rho)=(n+1)\alpha'.
\]
\end{enumerate}
\item
If $\rho$ is not self-dual, then the above observation on the slopes of the irreducible constituents of $\rho\otimes\rho$ implies that
 $\Sw(\Sym^2\rho)=\alpha'(n(n+1)/2-1)+\gamma$ and $\Sw(\wedge^2\rho)=\alpha'(n(n-1)/2-1)+\delta$ with $0\leq\gamma\leq\alpha'$ and  $0\leq\delta\leq\alpha'$.
Moreover, as $G=G_{1}$, $\delta$ cannot be $0$ (if $\delta$=0, then it means that $\wedge^{2}\rho$ contains the trivial character of $G$, hence $\rho$ is self-dual).
Thus we get
\[
\Sw(\Sym^2\rho)-\Sw(\wedge^2\rho)=n\alpha'+\gamma-\delta<(n+1)\alpha'.
\]
\end{enumerate}

\begin{lem}\label{lem:ZH=G}
We have $Z(G)H=G$.
\end{lem}

\begin{proof}
For $g\in G$ and $v\in V=H/Z$ with a lift $h\in H$, the commutator $ghg^{-1}h^{-1}$ does not depend on the choice of the lift $h$ and belongs to $Z$ since $G=G_{1}$ acts trivially on $V$.
Then $\Phi_{g}\colon V\rightarrow Z;\,v\mapsto ghg^{-1}h^{-1}$ is a linear homomorphism because we have, for $h,k\in H$, $ghkg^{-1}k^{-1}h^{-1}=(ghg^{-1}h^{-1})h(gkg^{-1}k^{-1})h^{-1}$ and $h$ commutes with $gkg^{-1}k^{-1}$ which is in $Z$.
Moreover, for $g,j\in G$, we have $gjhj^{-1}g^{-1}h^{-1}=g(jhj^{-1}h^{-1})g^{-1}(ghg^{-1}h^{-1})$, and $g$ commutes with $jhj^{-1}h^{-1}$, hence $gjhj^{-1}g^{-1}h^{-1}=(jhj^{-1}h^{-1})(ghg^{-1}h^{-1})$.
In other words, we have a group homomorphism 
\[
\Phi\colon G\rightarrow\Hom_{\F_{2}}(V,Z); \quad g\mapsto\Phi_{g}.
\]

Note that since $V$ is symplectic, the restriction of $\Phi$ to $H$ is surjective.
Let $J$ be the kernel of $\Phi$. 
The intersection of $H$ and $J$ is just $Z$, and $G=JH$.
Moreover, $J$ is the commutant of $H$ in $G$.
Thus, for any $j\in J$, $\rho(j)$ gives an $H$-automorphism of $\rho|_{H}$.
Since $\rho|_{H}$ is irreducible, this implies that $\rho(j)$ is a scalar multiplication by Schur's lemma.
Hence we get a homomorphism $\rho|_{J}\colon J\rightarrow \C^{\times}$.
Noting that $\rho$ is faithful, we see that $J$ is abelian.
Therefore, $J$ is central in $G$, which implies that $J=Z(G)$.
\end{proof}

\begin{lem}
We have $n\alpha\geq(n+1)\alpha'$
\end{lem}

\begin{proof}
The same proof as that of \cite[Th\'eor\`eme 1.8]{Hen80} works in the present setting.
For the sake of completeness, let us reproduce it here.

By Proposition \ref{prop:Heisenberg}, we have $\rho|_{H}\cong\Ind_{X}^{H}(\chi)$, where $X$ is a subgroup of $H$ such that $X$ contains $Z$ and $X/Z$ is a maximal totally isotropic subspace of $H/Z$ and $\chi$ is a character of $X$.
Let us put $\tilde{X}:=Z(G)X$ and $\tilde{Y}:=Z(G)Y$.
Then, since we have $G=Z(G)H$ by Lemma \ref{lem:ZH=G}, there exists an extension of $\tilde{\chi}$ of $\chi$ to $\tilde{X}$ satisfying $\rho\cong\Ind_{\tilde{X}}^{G}(\tilde{\chi})$.
By the induction formula for Swan exponents (Section \ref{sssec:Induction}; note that $[G\colon\tilde{X}]=n$), we have
\begin{align}\label{eq:Hen80-1}
\Sw(\rho)
=\Sw(\tilde{\chi})+v_{F}(\mfd_{E^{\tilde{X}}/F})-n+1,
\end{align}
where $E^{G'}$ denotes the fixed field of $G'$ in $E$ for any subgroup $G'$ of $G$.
\[
\xymatrix@R=10pt{
G\ar@{}[r]|{\supset} & H=G^{\alpha'} \\
\tilde{X}\ar@{}[r]|{\supset}\ar@{}[u]|{\cup} & X\ar@{}[u]|{\cup} \\
\tilde{Y} \ar@{}[r]|{\supset}\ar@{}[u]|{\cup}&Y \ar@{}[u]|{\cup}\\
Z(G) \ar@{}[r]|{\supset}\ar@{}[u]|{\cup}&Z=G^{\alpha} \ar@{}[u]|{\cup}
}
\]

To compute the right-hand side of \eqref{eq:Hen80-1}, we introduce a few notations.
For a finite Galois extension $L/K$ of non-Archimedean local fields, we put
\[
\alpha(L/K):=\inf\{u \mid \Gal(L/K)^{u}=\{1\}\},
\]
\[
\beta(L/K):=\inf\{v \mid \Gal(L/K)_{v}=\{1\}\}.
\]
Note that we have
\[
\varphi_{L/K}(\beta(L/K))=\alpha(L/K)
\quad\text{and}\quad
\psi_{L/K}(\alpha(L/K))=\beta(L/K),
\]
where $\varphi_{L/K}$ and $\psi_{L/K}$ are the Herbrand functions.
We remark that, in the case where $L/K$ is ramified, we have $\alpha(L/K)=\sup\{u \mid \Gal(L/K)^{u}\neq\{1\}\}$ and $\beta(L/K)=\sup\{v \mid \Gal(L/K)_{v}\neq\{1\}\}$ and these numbers are non-negative.

We fix a non-trivial element $s\in G/\tilde{X}\cong\Gal(E^{\tilde{X}}/F)$ and let $M\subset G$ be the preimage of $\langle s\rangle\subset G/\tilde{X}$ in $G$ and $N\subset \tilde{X}$ be the kernel of the character $\tilde{\chi}^{s-1}$.
Thus we have the chain 
\[
G \supset M \supset \tilde{X} \supset N \supset \tilde{Y}
\]
whose indices are $[G:M]=n/2$, $[M:\tilde{X}]=2$, $[\tilde{X}:N]=2$, and $[N:\tilde{Y}]=n/2$.
Then, by a result of Buhler (\cite[Proposition 3, page 31]{Buh78}), we have
\[
\Sw(\tilde{\chi})
\geq
\alpha(E^{N}/E^{\tilde{X}})+\beta(E^{\tilde{X}}/E^{M}).
\]
On the other hand, by \cite[Lemme 4.4]{Hen80}, we have
\[
v_{F}(\mfd_{E^{\tilde{X}}/F})
=
n(\alpha(E^{\tilde{X}}/F)+1)-\beta(E^{\tilde{X}}/F)-1.
\]
Hence, by \eqref{eq:Hen80-1}, we have
\begin{align}\label{ineq:Hen80-2}
\Sw(\rho)
\geq
\alpha(E^{N}/E^{\tilde{X}})+\beta(E^{\tilde{X}}/E^{M})+n\alpha(E^{\tilde{X}}/F)-\beta(E^{\tilde{X}}/F).
\end{align}

As the numbering of the lower ramification filtration is consistent with that for any subgroups (\cite[Chapter IV, Proposition 2]{Ser79}), we have
\[
\beta(E^{\tilde{X}}/E^{M})=\beta(E^{\tilde{X}}/F)
\]
(note that there is no jump between $H$ and $Z$ by definition).
Also, we have
\begin{align*}
\alpha(E^{N}/E^{\tilde{X}})
&=\varphi_{E^{N}/E^{\tilde{X}}}(\beta(E^{N}/E^{\tilde{X}}))\\
&=\varphi_{E^{N}/E^{\tilde{X}}}(\beta(E^{N}/F))\\
&=\varphi_{E^{N}/E^{\tilde{X}}}\circ\psi_{E^{N}/F}(\alpha(E^{N}/F))\\
&=\varphi_{E^{N}/E^{\tilde{X}}}\circ\psi_{E^{N}/E^{\tilde{X}}}\circ\psi_{E^{\tilde{X}}/F}(\alpha(E^{N}/F))
=\psi_{E^{\tilde{X}}/F}(\alpha(E^{N}/F)).
\end{align*}
By noting that the slope of $\psi_{E^{\tilde{X}}/F}$ is greater than or equal to $1$ and that $\alpha(E^{N}/F)=\alpha'$ (the upper numbering is consistent with taking a quotient of the Galois group; \cite[Chapter IV, Proposition 14]{Ser79}), we get $\alpha(E^{N}/E^{\tilde{X}})\geq \alpha'$.
Thus, finally noting that $\alpha(E^{\tilde{X}}/F)=\alpha'$, the above inequality \eqref{ineq:Hen80-2} implies that
\[
\Sw(\rho)
\geq
(n+1)\alpha'.
\]
Since we have $\Sw(\rho)=n\alpha$ (see Section \ref{sssec:upper}), this is nothing but the claimed inequality.
\end{proof}

\subsection{Approach via induction on the order of $G$}\label{subsec:induction-order}

We next investigate the conjectural inequality from the viewpoint of the induction on the order of $G$.

\begin{prop}\label{prop:orth-strict-inequality}
Suppose that Conjecture \ref{conj:Henniart} is true.
Then the inequality \eqref{ineq:Henniart} is strict when $\rho$ is irreducible orthogonal and non-trivial on $G_{1}$.
\end{prop}

\begin{proof}
We prove the statement by induction on the order of $G$.
Let $\rho$ be an irreducible orthogonal representation of $G$ on $V$ and assume that $G=G_{1}$ as usual.
We write $g:=|G|$ and $g_{i}:=|G_{i}|$.
Let $m\in\Z_{>0}$ be the integer such that $G_{m}\supsetneq G_{m+1}=\{1\}$.
By the definition of the Swan exponents, what we want to prove \eqref{ineq:Henniart} is that
\[
\sum_{i\geq1}\frac{g_i}{g}\bigl(\dim(\Sym^{2}V/(\Sym^{2}V)^{G_i})-\dim(\wedge^{2}V/(\wedge^{2}V)^{G_{i}})\bigr)
<
\sum_{i\geq1}\frac{g_i}{g}\dim(V/V^{G_{i}}).
\]
Since $\rho$ is irreducible, hence $V^{G_{i}}=0$ for any $i\leq m$, the above inequality is equivalent to 
\begin{align}\label{ineq:strict-orth}
\sum_{i=1}^{m}\frac{g_i}{g}\bigl(\dim((\Sym^{2}V)^{G_{i}})-\dim((\wedge^{2}V)^{G_{i}})\bigr)
>0.
\end{align}

Now let $a\in\Z_{>0}$ be the integer such that $G_{1}=\cdots=G_{a}\supsetneq G_{a+1}$.
Suppose that $H$ is a subgroup of $G$ containing $G_{a+1}$.
Let $h:=|H|$ and $h_{i}:=|H_{i}|$.
As the numbering of the lower ramification subgroups of $G$ is consistent with any subgroup of $G$ (\cite[Chapter IV, Proposition 2]{Ser79}), we have
\[
H=H_{1}=\cdots=H_{a}\supseteq H_{a+1} \,(=G_{a+1}) \supseteq\cdots\supseteq H_{m}\,(=G_{m})\supsetneq H_{m+1}=\{1\}.
\]
Let us take $H$ so that $[G:H]=2$.
In this case, there are three possibilities:
\begin{enumerate}
\item
$\rho|_{H}$ is irreducible, hence orthogonal;
\item
$\rho|_{H}$ has two irreducible constituents, both of which are orthogonal;
\item
$\rho|_{H}$ has two irreducible constituents, both of which are not self-dual.
\end{enumerate}

We first consider the cases (1) and (2).
We apply the induction hypothesis to $\rho|_{H}$ in the case (1) and to each irreducible constituent of $\rho|_{H}$ and then summing up them in the case (2).
Then, by the same reasoning as in the first paragraph, we obtain
\begin{align}\label{ineq:ineq-for-H}
\sum_{i=1}^{m}\frac{h_i}{h}\bigl(\dim((\Sym^{2}V)^{H_{i}})-\dim((\wedge^{2}V)^{H_{i}})\bigr)
>0
\end{align}
After multiplied by $h/g$, the summands of the left-hand side of \eqref{ineq:ineq-for-H} for $i>a$ are identical to those of \eqref{ineq:strict-orth}.
Thus, by noting that $G=G_{1}=\cdots=G_{a}$ and $H=H_{1}=\cdots=H_{a}$, it is enough to show that
\begin{align}\label{ineq:induction-argument}
\dim((\Sym^{2}V)^{G})-\dim((\wedge^{2}V)^{G})
\geq
\frac{h}{g}\bigl(\dim((\Sym^{2}V)^{H})-\dim((\wedge^{2}V)^{H})\bigr).
\end{align}
As $\rho$ is irreducible orthogonal, we have $\dim((\Sym^{2}V)^{G})=1$ and $\dim((\wedge^{2}V)^{G})=0$.
In the case (1), as $\rho|_{H}$ is irreducible orthogonal, we have $\dim((\Sym^{2}V)^{H})=1$ and $\dim((\wedge^{2}V)^{H})=0$.
In the case (2), as $\rho|_{H}$ has two irreducible orthogonal constituents, we have $\dim((\Sym^{2}V)^{H})=2$ and $\dim((\wedge^{2}V)^{H})=0$.
Since $h/g=1/2$, we get the assertion in both cases.

We next consider the case (3).
In this case, by applying Conjecture \ref{conj:Henniart} to each irreducible constituent of $\rho|_{H}$ and then summing up them, we obtain
\[
\sum_{i=1}^{m}\frac{h_i}{h}\bigl(\dim((\Sym^{2}V)^{H_{i}})-\dim((\wedge^{2}V)^{H_{i}})\bigr)
\geq0
\]
by the same reasoning as in the first paragraph.
Thus, by the same discussion as in the cases (1) and (2), it suffices to show that the inequality \eqref{ineq:induction-argument} holds without equality.
This follows from that $\dim((\Sym^{2}V)^{H})=0$ and $\dim((\wedge^{2}V)^{H})=0$ since two irreducible constituents of $\rho|_{H}$ are not self-dual.
\end{proof}

Thus it is also reasonable to expect the following.

\begin{conj}\label{conj:Henniart-ortho}
Assume $p=2$.
For any irreducible orthogonal Galois representation $\rho$ of $G$ which is not tame, we have
\[
\Sw(\Sym^{2}\rho)-\Sw(\wedge^{2}\rho)
<
\Sw(\rho).
\]
\end{conj}

The point here is that it could be possible to approach Conjecture \ref{conj:Henniart} by induction on the order of $G$ as performed in the proof of Proposition \ref{prop:orth-strict-inequality}.
At least, by a similar consideration to the above proof, we can conclude the following:
\begin{enumerate}
\item
If $\rho$ is orthogonal, then the inequality for $H$ implies the inequality for $G$, and even a strict one unless $\rho$ is induced from an orthogonal representation of $H$.
\item
If $\rho$ is not self-dual, then the inequality for $H$ implies the inequality for $G$ if $\rho|_{H}$ is not orthogonal.
\item
If $\rho$ is symplectic, then the inequality for $H$ implies the inequality for $G$ only if $\rho$ is induced from a symplectic representation of $H$.
\end{enumerate}

From this observation, we see that the obstacle for making the induction argument (on the order of $G$) work is the case where $\rho$ is symplectic.

\subsection{Approach via induction on the dimension of $\rho$}\label{subsec:induction-dimension}

Let us finish this section by giving some comments on the approach via the induction on the dimension of $\rho$.

Suppose that $\rho$ is irreducible, $G=G_{1}$, and $d:=\dim(\rho)>2$.
Again by the fact that any irreducible representation of a $p$-group is monomial, we may write $\rho=\Ind_{H}^{G}(\sigma)$, where $H$ is a subgroup of $G$ of index 2 and $\sigma$ is an irreducible representation of $H$.
Let $F'$ be the quadratic extension of $F$ corresponding to $H$ and $\omega$ be the quadratic character of $F^{\times}$ defining $F'/F$.
By Lemma \ref{lem:Ind-Swan} and Section \ref{sssec:upper}, we have
\[
\Sw(\rho)=\Sw(\sigma)+sd/2,
\]
where $s=\Sw(\omega)$.
On the other hand, by Proposition \ref{prop:Ind-Sym-Ext}, we have 
\[
\wedge^2\rho\cong\Ind_{H}^{G}(\wedge^{2}\sigma)\oplus\tau\otimes\omega
\quad\text{and}\quad
\Sym^2\rho\cong\Ind_{H}^{G}(\Sym^{2}\sigma)\oplus\tau,
\]
where $\tau$ is an extension of the representation $\sigma\otimes\sigma^{\gamma}$ of $H$ ($\gamma\in G\smallsetminus H$) to $G$ (see Proposition \ref{prop:Ind-Sym-Ext}).
Hence, again by Lemma \ref{lem:Ind-Swan}, we get
\[
\Sw(\Sym^2\rho)-\Sw(\wedge^2\rho)=\Sw(\Sym^{2}\sigma)-\Sw(\wedge^{2}\sigma)+sd/2+\Sw(\tau\otimes\omega)-\Sw(\tau).
\]
Therefore, to get the conjectural inequality, it suffices to show that
\[
\Sw(\Sym^{2}\sigma)-\Sw(\wedge^{2}\sigma)+\Sw(\tau\otimes\omega)-\Sw(\tau)
\leq
\Sw(\sigma).
\]

A favourable situation is when $\Sw(\tau\otimes\omega)=\Sw(\tau)$, because then the inequality for $\sigma$ implies the inequality for $\rho$, and similarly for the strict inequality.
This is the case if, for example, the slope of any irreducible constituent of $\tau$ is greater than $s$ by Lemma \ref{lem:Swan-twist}.
But we cannot insure this property in general.
Moreover, when $\sigma$ is not self-dual, hence $\sigma^\gamma$ is the contragredient of $\sigma$, $\sigma\otimes\sigma^\gamma$ contains the trivial character of $H$.
Thus $\tau$ contains at least one irreducible constituent whose slope is $0$ (i.e., the trivial character of $G$), hence smaller than $s$.
Therefore the equality $\Sw(\tau\otimes\omega)=\Sw(\tau)$ never holds in this case.
This means that we should have some stronger inequality for $\sigma$ to conclude the conjectural inequality for $\rho$ if $\sigma$ is not self-dual.

\section{Application to the local Langlands correspondence}\label{sec:LLC}

\subsection{Local Langlands correspondence}\label{subsec:LLC}
From now on, we assume that the characteristic of a non-Archimedean local field $F$ is zero.
We write $W_{F}$ for the Weil group of $F$.
Let $\G$ be a split connected reductive group over $F$.
Let $\hat{\G}$ denote the Langlands dual group of $\G$.
We say that a homomorphism $\phi\colon W_{F}\times\SL_{2}(\C)\rightarrow \hat{\G}$ is an \textit{$L$-parameter} of $\G$ if $\phi$ is smooth on $W_{F}$ and the restriction $\phi|_{\SL_{2}(\C)}\colon\SL_{2}(\C)\rightarrow\hat{\G}$ is algebraic.
We let $\Pi(\G)$ be the set of equivalence classes of irreducible smooth representations of $\G(F)$ and $\Phi(\G)$ the set of $\hat{\G}$-conjugacy classes of  $L$-parameters of $\G$.

The conjectural \textit{local Langlands correspondence} asserts that there exists a natural map
\[
\LLC_{\G}\colon\Pi(\G)\rightarrow\Phi(\G),
\]
with finite fibers.
In other words, by letting $\Pi_{\phi}^{\G}$ be the fiber of the map $\LLC_{\G}$ at an $L$-parameter $\phi$ (called an \textit{$L$-packet}), we have a natural partition
\[
\Pi(\G)
=
\bigsqcup_{\phi\in\Phi(\G)}\Pi_{\phi}^{\G},
\]
where each $\Pi_{\phi}^{\G}$ is finite.
Moreover, it is expected that each $\Pi_{\phi}^{\G}$ is equipped with a natural map $\Pi_{\phi}^{\G}\rightarrow\Irr(\mcS_{\phi})$, where $\Irr(\mcS_{\phi})$ is the set of irreducible characters of a certain finite group $\mcS_{\phi}$ associated to $\phi$ (see \cite[Section 1]{HII08} and also \cite[Section 6]{Art89} for the details).

The local Langlands correspondence has been established for several specific groups.
Especially, when $\G$ is $\GL_{N}$, the correspondence was constructed by Harris--Taylor \cite{HT01} and the first author \cite{Hen00}.
Also, for a certain class of classical groups, the correspondence has been established; for example, quasi-split symplectic and orthogonal groups by Arthur \cite{Art13}.
Our motivation is to seek an explicit description of the local Langlands correspondence for these groups.

\subsection{Hiraga--Ichino--Ikeda's formal degree conjecture}\label{subsec:FDC}
We next recall the \textit{formal degree conjecture} proposed by Hiraga--Ichino--Ikeda (\cite{HII08}).

\begin{conj}[Formal degree conjecture, {\cite[Conjecture 1.4]{HII08}}]\label{conj:FDC}
We assume the local Langlands correspondence for $\G$.
Suppose that $\pi\in\Pi(\G)$ is a discrete series representation whose $L$-parameter is $\phi$.
Then we have the following identity:
\[
\deg(\pi)
=
\frac{\langle1,\pi\rangle}{|\mathcal{S}_{\phi}^{\natural}|}\cdot|\gamma(0,\Ad\circ\phi,\psi_{F})|.
\]
Here, 
\begin{itemize}
\item
$\deg(\pi)$ is the formal degree of $\pi$ (see \cite{HII08} and also \cite{HII08-correction} for the normalization of the Haar measure used here),
\item
$\langle-,\pi\rangle$ denotes the irreducible character of $\mcS_{\phi}$ associated to $\pi$ via the above-mentioned map $\Pi_{\phi}^{\G}\rightarrow\Irr(\mcS_{\phi})$,
\item
$\mathcal{S}_{\phi}^{\natural}$ is a variant of the group $\mathcal{S}_{\phi}$ (see \cite[Section 1]{HII08} for the definition), 
\item
$\gamma(-,\Ad\circ\phi,\psi_{F})|$ denotes the $\gamma$-factor of the representation $\Ad\circ\phi$ of $W_{F}\times\SL_{2}(\C)$ with respect to a non-trivial additive character $\psi_{F}$, where $\Ad$ is the adjoint representation of $\hat{\G}$ on $\Lie(\hat{\G})/\Lie(Z(\hat{\G}))$.
\end{itemize}
\end{conj}

\subsection{Relation to Swan exponents}\label{subsec:FDC-Swan}
Now we explain why the above conjectures are related to the problem of our interest.
For this, let us first review the $\gamma$-factors for $L$-parameters.
For simplicity, we consider only the case where the $\SL_{2}(\C)$-part is trivial.
(See \cite[Section 2]{GR10} for the general case.)

We first consider the case of $\GL_{n}$.
In this case, an $L$-parameter with trivial $\SL_{2}(\C)$-part is nothing but an $n$-dimensional smooth semisimple representation of $W_{F}$.
Let $\phi$ be such a representation of $W_{F}$ on a $\C$-vector space $V$.
Then its $L$-factor is defined by
\[
L(\phi,s)
:=
\det(1-q^{-s}\cdot \phi(\Frob) \,\big\vert\, V^{\phi(I_{F})})^{-1},
\]
where $I_{F}$ denotes the inertia subgroup of $W_{F}$ and $\Frob$ is any lift of the geometric Frobenius.
Also, by fixing a non-trivial additive character $\psi_{F}$ of $F$, the $\varepsilon$-factor $\varepsilon(\phi,s,\psi_{F})$ is associated (Deligne \cite{Del73} and Tate \cite{Tat79}).
We do not recall the definition of the $\varepsilon$-factor, but note that it has the following relationship with the Artin exponent:
\[
\varepsilon(\phi,s,\psi_{F})
=
w\cdot q^{\Art(\phi)(\frac{1}{2}-s)},
\]
where $w$ is a complex number independent of $s$, which is a root of unity when $\rho$ is self-dual (see \cite[Section 2.3]{GR10}), and $\psi_{F}$ is taken to be of level zero, i.e., trivial on $\mcO_{F}$ but not on $\mfp_{F}^{-1}$.
Let us recall that the $\gamma$-factor is defined by using these local factors as follows:
\[
\gamma(\phi,s,\psi_{F})
:=
\varepsilon(\phi,s,\psi_{F})\cdot\frac{L(\phi^{\vee},1-s)}{L(\phi,s)},
\]
where $\phi^{\vee}$ is the contragredient representation of $\phi$.

We next consider the case where $\G$ is general.
In this case, by taking a finite-dimensional algebraic representation $R$ of $\hat{\G}$, we obtain an $L$-parameter $R\circ\phi$ of a general linear group.
Thus we can consider its $L$-factor, $\varepsilon$-factor, and $\gamma$-factor.
In the formal degree conjecture, $R$ is taken to be the adjoint representation of $\hat{\G}$.
Note that $\Ad$ is self-dual (consider the Killing form), hence so is $\Ad\circ\phi$.
For example, when $\G$ is one of $\GL_{N}$, $\SO_{N}$, and $\Sp_{2n}$, $\Ad\circ\phi$ is described as follows.
By composing $\phi$ with the standard representation of $\hat{\G}$, we may regard $\phi$ as a representation $\rho$ of $W_{F}$ (let us write $\rho$).
Then, $\Ad\circ\phi$ is isomorphic to $R\circ\rho$ as representations of $W_{F}$, where 
\[
R=
\begin{cases}
\Std\otimes\Std^{\vee}-\mathbbm{1}&\text{if $\G=\GL_{N}$,}\\
\Sym^{2}&\text{if $\G=\SO_{2n+1}$,}\\
\wedge^{2}&\text{if $\G=\Sp_{2n}$ or $\G=\SO_{2n}$.}
\end{cases}
\]

Now, by the above consideration, we can rewrite the identity predicted by the formal degree conjecture as follows:
\begin{align}\label{id:FDC}
\deg(\pi)
=
\frac{\langle1,\pi\rangle}{|\mathcal{S}_{\phi}^{\natural}|}\cdot q^{\frac{1}{2}\Art(\Ad\circ\phi)}\cdot\frac{|L(\Ad\circ\phi,1)|}{|L(\Ad\circ\phi,0)|}.
\end{align}
Our fundamental expectation is that establishing this identity is related to investigating the explicit local Langlands correspondence.
This is exactly the point where our motivation for computing the Swan exponent of $\Sym^{2}\rho$ or $\wedge^{2}\rho$ comes.

For example, let us consider the case where $\G$ is either $\SO_{N}$ or $\Sp_{2n}$.
In fact, Conjecture \ref{conj:FDC} in this case has been already solved (\cite{ILM17} for $\SO_{2n+1}$ and \cite{BP21-OWR} for $\SO_{2n}$ and $\Sp_{2n}$).
Suppose that an $L$-parameter $\phi$ is given ``explicitly'' in the sense that we have a description of the associated representation $\rho=\mathrm{std}\circ\phi$ of $W_{F}$.
Then, by computing the quantity on the right-hand side of \eqref{id:FDC}, we can access the formal degree of $\pi\in\Pi_{\phi}^{\G}$.
This enables us to narrow down the possibility of $\pi$.
Of course, here the roles of $\phi$ and $\pi$ can be swapped; we can also start from an explicitly given discrete series representation of $\G(F)$.

Conversely, we can also try to get the identity \eqref{id:FDC} by assuming the explicit local Langlands correspondence, i.e., assuming that we have an explicit description of an irreducible discrete series representation $\pi$ of $\G(F)$ and its $L$-parameter $\phi$.

\subsection{Some consequences for simple supercuspidal representations}\label{ssec:SSC}

Now we explain that the above strategy indeed works for \textit{simple supercuspidal representations}, in the sense of Gross--Reeder \cite{GR10}, of symplectic groups.

\subsubsection{From explicit LLC to FDC}\label{sssec:LLC-FDC}

\begin{cor}\label{cor:FDC-ssc-p-odd}
Simple supercuspidal representations of symplectic groups satisfy the formal degree conjecture when $p\neq2$.
\end{cor}

\begin{proof}
We utilize the result of \cite{Oi18-ssc}, which describes the $L$-parameter of a simple supercuspidal representation of $\Sp_{2n}(F)$ under the assumption that $p\neq2$.
As discussed in \cite[Section 9]{Oi18-ssc}, it suffices to show that, for any irreducible orthogonal representation $\rho$ of $W_{F}$ satisfying $\dim(\rho)=2n$ and $\Sw(\rho)=1$, we have $\Sw(\wedge^{2}\rho)=n-1$.
This is a special case of Corollary \ref{cor:p-odd}.
\end{proof}

\begin{rem}
In \cite{Oi18-ssc}, the second author proved the equality $\Sw(\wedge^{2}\rho)=n-1$, hence obtained Corollary \ref{cor:FDC-ssc-p-odd}, in the case where $p\nmid 2n$ or $p=p^{e}\cdot n'$ for $e\in\Z_{>0}$ and $n'\in\Z_{>0}$ satisfying $n'|(p-1)$.
In \cite{Mie21}, Mieda obtained Corollary \ref{cor:FDC-ssc-p-odd} for any $p\neq 2$ via a completely different method.
On the other hand, as mentioned above, it was announced by Beuzart-Plessis that the formal degree conjecture for $\Sp_{2n}$ has been solved (see \cite{BP21-OWR}).
Hence Corollary \ref{cor:FDC-ssc-p-odd} is not new in any case.
But we would like to emphasize that our proof presented above is new.
\end{rem}

\subsubsection{From FDC to explicit LLC}\label{sssec:FDC-LLC}

We next consider the case where $p=2$.
In this case, we determined explicitly the $L$-parameter of a simple supercuspidal representation of $\Sp_{2n}(F)$ in \cite{HO22}.
However, we cannot deduce the formal degree conjecture for simple supercuspidal representations of $\Sp_{2n}(F)$ from the description in \textit{loc.\ cit.}\ as we did in Section \ref{sssec:LLC-FDC}.
Rather, in contrast to the case where $p\neq2$, we utilized the formal degree conjecture in order to determine the $L$-parameter.
Let us just present the outline of the arguments of \textit{loc.\ cit.}\ here.

The point is starting from an irreducible $(2n+1)$-dimensional orthogonal representation $\rho$ of $W_{F}$ whose Swan exponent is $1$.
Then it associates an irreducible supercuspidal representation $\pi$ of $\Sp_{2n}(F)$.
On the other hand, it also associates an irreducible self-dual supercuspidal representation $\Pi$ of $\GL_{2n+1}(F)$, which is known to be simple supercuspidal.
By utilizing the twisted endoscopic character relation between $\Pi$ and $\pi$, we can prove that $\pi$ is either depth-zero supercuspidal or simple supercuspidal (\cite[Corollary 4.5]{HO22}).
Since we know that $\Sw(\wedge^{2}\rho)=n$ (\cite[Proposition 4.12]{HO22}, or more generally, Corollary \ref{cor:ssc-p-2} of this paper), the formal degree conjecture tells us the value of the formal degree of $\pi$.
In fact, this information is enough for concluding that $\pi$ is not depth-zero, hence simple supercuspidal (see \cite[A.4]{Hen23}).
Once we see that $\pi$ is simple supercuspidal, it is not difficult to determine it exactly again by using the twisted endoscopic character relation.

\begin{rem}
Although we only mentioned the case of $\Sp_{2n}$, it is also possible to establish the explicit local Langlands correspondence for simple supercuspidal representations of split special orthogonal groups by utilizing the formal degree conjecture effectively.
That is the content of \cite{AHKO23}.
\end{rem}

\section{Simple supercuspidals $L$-parameters for $\Sp_{6}(\Q_2)$ have value in $G_2(\C)$}\label{sec:G2}

This section, a remark really, is prompted by a question of Gordan Savin to the authors.
After we had submitted our previous paper \cite{HO22} mentioned in Section \ref{sssec:FDC-LLC}, Savin asked us if the $L$-parameter of a simple supercuspidal representation of $\Sp_6(\Q_2)$, which takes values in $\SO_{7}(\C)$, actually can be conjugated to take values in the subgroup $G_2(\C)$ of $\SO_7(\C)$.

That question is related to the analysis of the parameter of simple supercuspidal representations of $G_2(\Q_2)$ in \cite[Section 4]{KLS10}.
In \textit{loc.\ cit.}, a subgroup $I$ of $G_2(\C)$ is introduced.
The group $I$ is a semi-direct product of a normal subgroup $J$, which is isomorphic to the additive $2$-group $\F_8$ (where $\F_8$ is a field with $8$ elements), with a subgroup which is itself a semi-direct product of $\F_8^{\times}$ and $\Gal(\F_8/\F_2)$ (acting naturally on $\F_8^{\times}$), with their natural actions on $\F_8$. 
\[
\xymatrix{
1\ar[r] & J \,(\cong \F_{8})\ar[r] &I \ar[r]&I/J \, (\cong \F_{8}^{\times}\rtimes\Gal(\F_{8}/\F_{2})) \ar[r]\ar@/^15pt/^-{\text{split}}[l] & 1
}
\]
It is shown that the inclusion of $I$ into $G_2(\C)$, itself inside $\GL_7(\C)$, gives the unique irreducible faithful self-dual (hence orthogonal) representation of $I$ of dimension $7$. 
It is also explicited in \textit{loc.\ cit.} how the group $I$ is a quotient of $W_{\Q_2}$.

Now take a simple supercuspidal representation $\pi$ of $\Sp_6(\Q_2)$, and let $\sigma$ be the corresponding representation of $W_{\Q_2}$, which is an orthogonal irreducible representation of dimension $7$.

\begin{prop}
The representation $\sigma$ is the same as the representation described after Proposition 4.2 in \textit{loc.\ cit.}
\end{prop}

\begin{proof}
By \cite{HO22} (cf.\ \cite{BH11}), the representation $\sigma$ is given by $\Ind_{K/\Q_{2}}(\chi)$, where $K/\Q_{2}$ is the tame totally ramified extension of degree $7$ generated by a $7$th root $\varpi$ of $2$ (remark that $K$ is unique up to isomorphism) and $\chi$ is a quadratic character of $K^{\times}$ of Swan exponent $1$. 
Note that such $\chi$ must be as follows:
\begin{itemize}
\item
$\chi(\varpi)$ is either $1$ or $-1$, 
\item
$\chi|_{U^{2}_{K}}$ is trivial, and
\item
$\chi|_{U^{1}_{K}}$ is the unique non-trivial character (say $\chi_{1}$) of $U^1_K$. 
\end{itemize}
Under the identification
\[
U^{1}_{K}/U^{2}_{K}\cong \F_{2}\colon 1+\varpi x \mapsto \overline{x} 
\quad(x\in \mcO_K),
\]
where $\overline{x}$ denotes the modulo $\mfp_{K}$ reduction, we transform $\chi_{1}$ into the non-trivial character of $\F_2$. 
Moreover, since we have $\det(\sigma)\cong \delta_{K/\Q_{2}}\cdot(\chi|_{\Q_{2}^{\times}})$, where $\delta_{K/\Q_{2}}:=\det(\Ind_{K/\Q_{2}}\mathbbm{1})$, the triviality of $\det(\sigma)$ implies that $\chi|_{\Q_{2}^{\times}}=\delta_{K/\Q_{2}}$ (see \cite[29.2]{BH06}).
By \cite[(10.1.6) Proposition]{BF83}, $\delta_{K/\Q_{2}}$ is unramified and takes value the Jacobi symbol $(\frac{2}{7})$ at $2$.
But $(\frac{2}{7})=1$ since $3^2\equiv2 \pmod{7}$.
Hence $\chi|_{\Q_{2}^{\times}}$ is trivial.
In particular, $\chi(\varpi)$ must be $1$.
Conversely, $\chi|_{\Q_{2}^{\times}}$ is trivial when $\chi(\varpi)=1$.

On the other hand, \cite{KLS10} introduces the Galois closure $L$ of $K$, which is the splitting field of $K/\Q_2$.
Then the $21$-dimensional representation $\rho$ of $W_{\Q_2}$ is introduced, which is given by $\Ind_{L/\Q_{2}}(\eta)$, where $\eta$ is a quadratic character of $L^{\times}$ taking value $1$ at the uniformizer $\varpi$ and giving on restriction to $U^1_L$ any non-trivial character trivial on $U^2_L$ (there are $7$ of them, any of which induces to $\rho$). 
Then $\rho$ splits into the sum of three irreducible representations of dimension $7$, only one of which is self-dual.
 
But we can take $\eta=\chi\circ N_{L/K}$, which shows that $\sigma$ is the self-dual irreducible component of $\rho$. 
That proves what we claimed.
\end{proof}

The above arguments can be generalized to any $2$-adic field $F$ as a base field. 

\begin{prop}
Let $\pi$ be a simple supercuspidal representation of $\Sp_6(F)$ and $\sigma$ be the corresponding representation of $W_{F}$, which is an orthogonal irreducible representation of dimension $7$.
Then $\sigma$ takes values (up to conjugation) in the subgroup $I$ of $G_2(\C)$, and its image is then either $I$ or its subgroup of index $3$.
\end{prop}

\begin{proof}
Again by \cite{HO22} (cf.\ \cite{BH11}), $\sigma$ is given by $\Ind_{K/F}(\chi)$, where $K/F$ is a totally ramified extension of degree $7$ and $\chi$ is a quadratic character of $K^{\times}$ with $\Sw(\chi)=1$. 

Let $q_{F}$ be the cardinality of $k_{F}$.
We first consider the case where $7\nmid (q_{F}-1)$.
In this case, there is a unique such extension $K$ up to isomorphism, generated by a $7$th root $\Pi$ of a uniformizer of $F$.
Its Galois closure $K'/F$ is $KF'$ where $F'/F$ is the unramified extension of degree $3$ obtained by adjoining $7$th roots of unity. 
Then we have 
\[
\Gal(K'/F)\cong\Gal(K'/F')\rtimes\Gal(K'/K), 
\]
where the inertia subgroup $\Gal(K'/F')$ is cyclic of order $7$. 
The Frobenius element of $\Gal(K'/K)$ acts on $\Gal(K'/F')$ by taking $q_{F}$-th powers, an order $3$ automorphism of $\Gal(K'/F')$. 
\[
\xymatrix{
&K' \ar@{-}^-{\mathrm{Gal}}[dd]&\\
F'\ar@{-}^-{\text{tame of deg $7$}}[ru]&&K\ar@{-}_-{\text{ur.\ of deg $3$}}[lu] \\
&F\ar@{-}^-{\text{ur.\ of deg $3$}}[lu]\ar@{-}_-{\text{tame of deg $7$}}[ru]&
}
\]

The possible inducing characters of $K^{\times}$, to get by induction to $W_{F}$ an irreducible special orthogonal representation with Swan exponent $1$, are the quadratic characters $\chi$ non-trivial on $U_F^1$ but trivial
on $U_F^2$, and taking the value $1$ on the uniformizer $\Pi$. 
There are $q_F-1$ of them, as it should.
Consider the character $\chi'=\chi\circ\Nr_{K'/K}$ of $K^{\prime\times}$, which induces to $W_{F'}$ the restriction $\sigma'$ of $\sigma$.
We identify the residue field $k_{K'}$ of $K'$ with $U_{K'}^1/U_{K'}^2$ by the map
\[
U_{K'}^1/U_{K'}^2 \cong k_{K'} \colon 1+\Pi x \mapsto \overline{x} \quad(x\in\mcO_{K'}).
\]

The action of $\Gal(K'/F')$ is given by multiplication by $7$th roots of $1$, reflecting the action on $\Pi$, and $\Gal(K'/K)$ acts by its natural action on $k_{K'}=k_{F'}$.
The field $k_{F'}$ is an extension of $\F_8$, and the action of $\Gal(K'/F')$ by $7$th roots of $1$ is an action via a morphism into $\F_8^{\times}$. 
It follows that the restriction of $\sigma$ to $W_{K'}$, which is the direct sum of all conjugates of $\chi'$, factors through the quotient of $k_{F'}$ by an $\F_8$ hyperplane, and in particular its image is a line over $\F_8$. 
It follows again that the image of $\sigma$ is isomorphic to $I$.

We next consider the case where $7|(q_F-1)$.
In this case, the situation is simpler, as $K/F$ is cyclic of degree $7$. 
A similar, but simpler, analysis goes through to show that the image of $\sigma$ is isomorphic to the index $3$ subgroup of $I$.
\end{proof}

\begin{rem}
An issue that might come up is that the group $I$ or its index $3$ subgroup $I_0$, which have only one self-dual irreducible representation $V$ of dimension $7$, may still have several embedding into $G_2(\C)$, up to conjugation. It is however not the case, because there is only one line of fixed vectors in $W=\wedge^{3}V$: we know there is one because $I$ is obtained in \cite{KLS10} as a subgroup of $G_2(\C)$, and if you look at the representation of $J$ in $W$, the fixed point set has dimension $7$, and as a representation of $I_0/J$ is the direct sum of all characters.
\end{rem}

\appendix
\section{Some basic facts on representations of finite groups}

In this section, we collect some basic facts on representations of finite groups.
The facts introduced in this section might be well-known, but we give proofs for the sake of completeness.

\subsection{Symmetric and Exterior squares of induced representations}

Let $G$ be a finite group and $H$ a subgroup of $G$ of index $2$.
We fix an element $\gamma$ of $G\smallsetminus H$.
Let $\sigma$ be a representation of $H$.
If we put
\[
\begin{cases}
\tau(h)(v_{1}\otimes v_{2}):=(\sigma(h)v_{1})\otimes(\sigma^{\gamma}(h)v_{2})\\
\tau(\gamma h)(v_{1}\otimes v_{2}):=(\sigma^{\gamma}(h)v_{2})\otimes(\sigma(\gamma^{2})\sigma(h)v_{1})
\end{cases}
\]
for any $h\in H$, then it can be easily checked that $\tau$ is a representation of $G$ which extends the representation $\sigma\otimes\sigma^\gamma$ of $H$.

\begin{prop}\label{prop:Ind-Sym-Ext}
We put $\rho:=\Ind_{H}^{G}(\sigma)$.
Let $\omega$ be the non-trivial character of $G/H$.
Then we have
\begin{enumerate}
\item
$\wedge^2\rho\cong\Ind_{H}^{G}(\wedge^{2}\sigma)\oplus\tau\otimes\omega$.
\item
$\Sym^2\rho\cong\Ind_{H}^{G}(\Sym^{2}\sigma)\oplus\tau$.
\end{enumerate}
\end{prop}

We first show some lemmas.

\begin{lem}\label{lem:tensor-twisted-trace}
Let $V_{1}$ and $V_{2}$ be finite dimensional $\C$-vector spaces equipped with isomorphisms $I_{1}\colon V_{1}\rightarrow V_{2}$ and $I_{2}\colon V_{2}\rightarrow V_{1}$.
If we define an automorphism $I_{1}\otimes I_{2}$ of $V_{1}\otimes V_{2}$ by $I_{1}\otimes I_{2}(v_{1}\otimes v_{2}):=I_{2}(v_{2})\otimes I_{1}(v_{1})$, then we have 
\[
\Tr(I_{1}\otimes I_{2}\mid V_{1}\otimes V_{2})=\Tr(I_{2}\circ I_{1}\mid V_{1}).
\]
\end{lem}

\begin{proof}
We take a $\C$-basis $\{e_{1},\ldots,e_{n}\}$ of $V_{1}$ and define a $\C$-basis $\{e'_{1},\ldots,e'_{n}\}$ of $V_{2}$ by $e'_{i}:=I_{1}(e_{i})$.
Then $\{e_{j}\otimes e'_{j}\}_{1\leq i,j \leq n}$ forms a $\C$-basis of $V_{1}\otimes V_{2}$.
Since we have
\[
I_{1}\otimes I_{2}(e_{i}\otimes e'_{j})
=I_{2}(e'_{j})\otimes I_{1}(e_{i})
=(I_{2}\circ I_{1})(e_{j})\otimes e'_{i},
\]
$e_{i}\otimes e'_{j}$ contributes to the trace of $I_{1}\otimes I_{2}$ only when $i=j$ (see the matrix in the basis $\{e_{j}\otimes e'_{j}\}_{1\leq i,j \leq n}$).
Moreover, the contribution is exactly that of $e_{i}$ to the trace of $I_{2}\circ I_{1}$ (seen in the basis $\{e_{i}\}_{1\leq i\leq n}$).
\end{proof}

\begin{lem}\label{lem:twisted-trace}
\begin{enumerate}
\item
We have $\Tr(\sigma)(g^{2})=\Tr(\sigma^{\gamma})(g^{2})$ for any $g\in G$, where $\gamma$ is a(ny) element of $G\smallsetminus H$.
\item
We have 
\[
\Tr(\tau)(g)
=
\begin{cases}
\Tr(\sigma)(g)\cdot \Tr(\sigma^{\gamma})(g) & \text{if $g\in H$,}\\
\Tr(\sigma)(g^{2}) & \text{if $g\in G\smallsetminus H$.}
\end{cases}
\]
\end{enumerate}
\end{lem}

\begin{proof}
We show (1).
When $g\in G\smallsetminus H$, by writing $g=h\gamma$ with $h\in H$, we get $\gamma g^{2}\gamma^{-1}=\gamma h \gamma h=h^{-1}g^{2}h$.
In particular, $\gamma g^{2}\gamma^{-1}$ and $g^{2}$ are conjugate in $H$.
Hence we have $\Tr(\sigma^{\gamma})(g^{2})=\Tr(\sigma)(g^{2})$.

We next show (2).
As $\tau$ is an extension of $\sigma\otimes\sigma^{\gamma}$, the equality $\Tr(\tau)(g)=\Tr(\sigma)(g)\cdot\Tr(\sigma^{\gamma})(g)$ for $g\in H$ is obvious.
For the case where $g=\gamma h\in G\smallsetminus H$, by Lemma \ref{lem:tensor-twisted-trace}, we get
\[
\Tr(\tau)(g)
=\Tr(\sigma(\gamma^{2})\sigma(h)\sigma^{\gamma}(h))
=\Tr(\sigma)(\gamma^{2}h\gamma h\gamma^{-1})
=\Tr(\sigma^{\gamma})(g^{2})
=\Tr(\sigma)(g^{2}),
\]
where we used (1) in the last equality.
\end{proof}

\begin{proof}[Proof of Proposition \ref{prop:Ind-Sym-Ext}]
By noting that $\rho\otimes\rho\cong \Sym^{2}\rho\oplus\wedge^{2}\rho$ and $\Psi^{2}\rho\cong \Sym^{2}\rho-\wedge^{2}\rho$ (and that the same is true for $\sigma$), it suffices to show that 
\begin{enumerate}
\item[(1')]
$\rho\otimes\rho\cong\Ind_{H}^{G}(\sigma\otimes\sigma)\oplus\tau\oplus\tau\otimes\omega$ and 
\item[(2')]
$\Psi^2\rho\cong\Ind_{H}^{G}(\Psi^{2}\sigma)\oplus\tau-\tau\otimes\omega$.
\end{enumerate}

We first note that the Frobenius formula for induced representations implies that
\[
\Tr(\rho)(g)
=
\begin{cases}
\Tr(\sigma)(g)+\Tr(\sigma^{\gamma})(g) & \text{if $g\in H$,}\\
0 & \text{if $g\in G\smallsetminus H$.}
\end{cases}
\]
Thus, we have
\[
\Tr(\rho\otimes\rho)(g)
=
\Tr(\rho)(g)^{2}
=
\begin{cases}
(\Tr(\sigma)(g)+\Tr(\sigma^{\gamma})(g))^{2} & \text{if $g\in H$,}\\
0 & \text{if $g\in G\smallsetminus H$,}
\end{cases}
\]
and 
\[
\Tr(\Psi^{2}\rho)(g)
=
\Tr(\rho)(g^{2})
=
\Tr(\sigma)(g^{2})+\Tr(\sigma^{\gamma})(g^{2})
\]
where we note that $g^{2}$ always belongs to $H$ in the last equality.
Similarly, we have
\begin{align*}
\Tr(\Ind_{H}^{G}(\sigma\otimes\sigma))(g)
&=
\begin{cases}
\Tr(\sigma\otimes\sigma)(g)+\Tr((\sigma\otimes\sigma)^{\gamma})(g) & \text{if $g\in H$,}\\
0 & \text{if $g\in G\smallsetminus H$,}
\end{cases}\\
&=
\begin{cases}
\Tr(\sigma)(g)^{2}+\Tr(\sigma^{\gamma})(g)^{2} & \text{if $g\in H$,}\\
0 & \text{if $g\in G\smallsetminus H$,}
\end{cases}
\end{align*}
and 
\begin{align*}
\Tr(\Ind_{H}^{G}(\Psi^{2}\sigma))(g)
&=
\begin{cases}
\Tr(\Psi^{2}\sigma)(g)+\Tr(\Psi^{2}\sigma^{\gamma})(g) & \text{if $g\in H$,}\\
0 & \text{if $g\in G\smallsetminus H$,}
\end{cases}\\
&=
\begin{cases}
\Tr(\sigma)(g^{2})+\Tr(\sigma^{\gamma})(g^{2}) & \text{if $g\in H$,}\\
0 & \text{if $g\in G\smallsetminus H$.}
\end{cases}
\end{align*}
Thus, the identities (1') and (2') follows by using Lemma \ref{lem:twisted-trace}.
\end{proof}

\subsection{Faithful irreducible representation of a $p$-group}

\begin{lem}\label{lem:p-group-faithful-center}
Let $G$ be a $p$-group.
If $G$ has a faithful irreducible representation, then the center $Z(G)$ of $G$ is a non-trivial cyclic group.
\end{lem}

\begin{proof}
Suppose that $\rho$ is a faithful irreducible representation of $G$.
It is well-known that any $p$-group has a non-trivial center (see, e.g., \cite[Theorem 14 in Chapter 8]{Ser77}).
Thus $Z(G)$ is a non-trivial finite abelian $p$-group.
If we let $\omega_{\rho}$ be the central character of $\rho$, then $\omega_{\rho}$ is also faithful.
However, $Z(G)$ must be cyclic so that $Z(G)$ has a faithful $1$-dimensional character.
\end{proof}

\subsection{Representations of a certain $p$-group}\label{subsec:Heisenberg}

Let $p$ be a prime number.
Let $H$ be a finite $p$-group and $Z(H)$ its center.
We assume that $H$ is non-abelian and that there exists a subgroup $Z\subset Z(H)$ satisfying the following conditions:
\begin{enumerate}
\item
$Z$ is cyclic of order $p$, and
\item
$H/Z$ is an elementary abelian $p$-group.
\end{enumerate}
We remark that $G$ is said to be a \textit{extra-special $p$-group} if $Z$ can be taken to be $Z(H)$.

We note that, by the condition (2), the image of the commutator map $H\times H\rightarrow H\colon (x,y)\mapsto xyx^{-1}y^{-1}$ lies in $Z$.
Thus, by fixing an identification $Z\cong\F_{p}$, we obtain a symplectic form (i.e., non-degenerate alternating bilinear form) on $H/Z(H)$
\[
\langle-,-\rangle\colon H/Z(H)\times H/Z(H) \rightarrow \F_{p}.
\]
We take a maximal totally isotropic subspace $U\subset H/Z(H)$ with respect to $\langle-,-\rangle$.
We let $H'$ be the preimage of $U$ in $H$, which is abelian.

The following proposition is given in \cite[16.4 Lemma 2]{BH06} in the case where $Z=Z(H)$ (in fact, it is allowed that $Z=Z(H)$ is any finite cyclic group in \textit{loc.\ cit.}).

\begin{prop}\label{prop:Heisenberg}
For any character $\chi$ of $Z(H)$ whose restriction to $Z$ is non-trivial, there uniquely (up to isomorphism) exists an irreducible representation $\rho_{\chi}$ of $H$ on which $Z(H)$ acts via $\chi$.
Moreover, $\rho_{\chi}$ is explicitly given by $\Ind_{H'}^{H}(\chi')$ for any character $\chi'$ of $H'$ extending $\chi$ (note that such $\chi'$ can be always taken since $H'$ is abelian).
\end{prop}

\begin{proof}
We just reproduce the proof of \cite[16.4 Lemma 2]{BH06} in our setting.

The center $Z(H)$ obviously acts on $\Ind_{H'}^{H}(\chi')$ via $\chi$.
Thus let us next check that $\Ind_{H'}^{H}(\chi')$ is irreducible.
Since $H'$ is normal in $H$, it suffices to show that any element $h\in H$ satisfying $(\chi')^{h}=\chi'$ lies in $H'$.
If $h\in H$ is such an element, we have $\chi'(hh'h^{-1}h^{\prime-1})=1$ for any $h'\in H'$.
Since $\chi'$ extends the character $\chi$ of $Z$, which is faithful, this implies that $hh'h^{-1}h^{\prime-1}=1$ for any $h'\in H'$.
As $H'$ is the preimage of the maximal totally isotropic subspace $U$ of $H/Z(H)$ with respect to $\langle-,-\rangle$, we see that $h$ must belongs to $H'$.

Let us finally show the uniqueness of $\rho_{\chi}$ by a counting argument.
For this, we let $|H/Z(H)|=p^{2d}$ and $|Z(H)/Z|=p^{r}$, hence $|H|=p^{2d+r+1}$.
Note that
\begin{align*}
|H|
&=\sum_{\text{$\rho$: irr.\ rep.\ of $H$}} (\dim\rho)^{2}\\
&=\sum_{\begin{subarray}{c}\text{$\rho$: irr.\ rep.\ of $H$} \\ \text{trivial on $Z$} \end{subarray}} (\dim\rho)^{2} + \sum_{\begin{subarray}{c}\text{$\rho$: irr.\ rep.\ of $H$} \\ \text{non-trivial on $Z$} \end{subarray}} (\dim\rho)^{2}.
\end{align*}
Since $H/Z$ is abelian, the first term is simply given by $|H/Z|=p^{2d+r}$.
On the other hand, by noting that $\dim(\Ind_{H'}^{H}(\chi'))^{2}=p^{2d}$ and that the number of characters of $Z(H)$ non-trivial on $Z$ is given by $p^{r}(p-1)$, we see that the second term is bounded below by $p^{2d+r}(p-1)$.
Thus we get $|H|\geq p^{2d+r+1}$.
However, we already know that $|H|=p^{2d+r+1}$.
Therefore, we conclude that the representations $\Ind_{H'}^{H}(\chi')$ exhausts all irreducible representations of $H$ whose restrictions to $Z$ are non-trivial.
\end{proof}

\end{document}